\documentclass[11pt]{amsart}

\usepackage{times} 
\usepackage{enumerate}
\usepackage{esint}
\usepackage{pgf}
\usepackage{tikz}
\usetikzlibrary{patterns,shapes,calc,through,
arrows,fadings,decorations.pathreplacing,intersections}

\def\eps{\varepsilon}
\def\real{\mathbb{R}}
\def\supp{\mathop{\mbox{\normalfont supp}}\nolimits}

\def\essinf{\mathop{\mbox{\normalfont ess inf}}\limits}
\def\esssup{\mathop{\mbox{\normalfont ess sup}}\limits}
\def\tiende{\mathop{\longrightarrow}\limits}
\def\rad{\mathop{\emph{rad}}}
\def\loc{\mbox{\scriptsize{\normalfont{loc}}}}

\makeatletter\renewcommand\theenumi{\@roman\c@enumi}\makeatother

\newtheorem{teo}{Theorem}[section]

\newtheorem{lema}[teo]{Lemma}
\newtheorem{prop}[teo]{Proposition}

\theoremstyle{definition}
\newtheorem*{xrem}{Remark}

\numberwithin{equation}{section}

\let\oldmarginpar\marginpar
\renewcommand\marginpar[1]{\-\oldmarginpar[\raggedleft\footnotesize #1]
{\raggedright\footnotesize #1}}

\begin{document}

\title[Localization and dimension free estimates for maximal functions]{Localization and dimension free estimates for maximal functions.}

\author[A. Criado]{Alberto Criado}
\address{Alberto Criado \newline\indent Departamento de Matem\'aticas, Universidad Aut\'onoma de Madrid, 28049 Madrid, Spain}
\email{alberto.criado@uam.es}
\thanks{Authors supported by DGU grant MTM2010-16518.}

\author[F. Soria]{Fernando Soria}
\address{Fernando Soria \newline\indent Departamento de Matem\'aticas and Instituto de Ciencias Matem\'aticas CSIC-UAM-UCM-UC3M, Universidad Aut\'onoma de Madrid, 28049 Madrid, Spain}
\email{fernando.soria@uam.es}

\begin{abstract}
In the recent paper [J. Funct. Anal. {\bf 259} (2010)], Naor and Tao introduce a new class of measures with a so-called  micro-doubling  property and present, via martingale theory and probability methods, a localization theorem for the associated maximal functions. As a consequence they obtain a weak type estimate in a general abstract setting for these maximal functions that is reminiscent of the `$n\log n$ result' of Stein and Str\"omberg in Euclidean spaces. 

The purpose of this work is twofold. First we introduce a new localization principle that localizes not only in the time-dilation parameter but also in space. The proof uses standard covering lemmas and selection processes. Second, we show that a uniform condition for micro-doubling in the Euclidean spaces provides indeed dimension free estimates for their maximal functions in all $L^p$ with $p>1$. This is done introducing a new technique that allows to differentiate through dimensions.
\end{abstract}


\maketitle

\section{Introduction}

\bigskip

We say that $(X,d,\mu)$ is a metric measure space if $(X,d)$ is a separable metric space and $\mu$ a Radon measure on it. We denote by $B(x,r)$ the open ball centered at $x$ with radius $r$ with respect to the metric $d$, that is
\[
B(x,r)=\{y\in X: d(x,y) < r\}.
\]
We will assume that the measure $\mu$ is non-degenerate. This means that any ball with positive radius has non-zero measure. Given $T\subset (0,\infty)$, for a locally integrable function $f$ over $X$ we define the following centered maximal operator
\[
M_Tf(x)=\sup_{r\in T} \frac1{\mu(B(x,r))}\int_{B(x,r)}|f(y)|\,d\mu(y).
\]
When $T=(0,\infty)$, the operator $M=M_T$ represents the usual Hardy-Littlewood maximal operator. As in the case of evolution equations and semigroup theory we can think of $T$ as a set of times and then $M_T$ is the maximal operator on them.

\bigskip

We will also assume that the measure $\mu$ satisfies a doubling property. This simply says that the measure of a ball is comparable with the measure of certain dilation of it. The classical way of expressing the doubling property uses the dilation factor 2. That is, $\mu$ is doubling if there exists a constant $K>0$ so that for each $x\in X$ and $R>0$ one has $\mu(B(x,2R))\leq K \mu(B(x,R))$.

\bigskip

In these hypotheses one can reproduce the argument of Vitali's covering lem\-ma to show that $M_T$ is weakly bounded on $L^1(\mu)$. That is, there exist a constant $c_{\mu,1}>0$ so that
\[
\mu(\{x\in X: M_Tf(x)>\lambda\}) \leq \frac {c_{\mu,1}}\lambda \int_X |f|d\mu,
\]
for all $\lambda>0$ and all locally integrable $f$ over $X$. Since $\|M_T f\|_{L^\infty(\mu)}\leq \|f\|_{L^\infty(\mu)}$, by interpolation one obtains the $L^p(\mu)$ bounds
\[
\|M_T f\|_{L^p(\mu)}\leq C_{\mu,p} \|f\|_{L^p(\mu)},
\]
for all $p>1$. The order of magnitude of the constants $c_{\mu,1}$ and $C_{\mu,p}$ that we obtain is essentially that of the doubling constant $K$ and $K^{1/p}p/(p-1)$ respectively.

\bigskip

In the case that $X=\real^n$ and $\mu=m_n$ is Lebesgue measure, the value of the doubling constant $K$ is $2^n$. Hence the aforementioned constants $c_{m_n,1}$ and $C_{m_n,p}$ grow exponentially to infinity with the dimension. It has been a matter of interest to know if these constants can be bounded uniformly in dimension. E.M. Stein \cite{Stein1} (detailed proof in \cite{SteinStromberg}) proved this for $C_{m_n,p}$ with $p>1$ when considering the Euclidean metric. J. Bourgain \cite{Bourgain1}, \cite{Bourgain2}, \cite{Bourgain3} and independently A. Carbery \cite{Carbery}, showed that for the metric given by a norm, uniform bounds hold if $p>3/2$. This was extended to all the range $p>1$ by D. M\"uller in \cite{Muller} for the maximal functions associated with the $\ell^q$ metrics, with $1\leq q<\infty$. Observe that this excludes the case $q=\infty$, where the `balls' are cubes with sides parallel to the coordinate axes. Recently J. Bourgain has solved this case by giving uniform bounds on $L^p$ for all $p>1$ (see \cite{Bourgain5}).

\bigskip

The case $p=1$ is more complicated. J.M. Aldaz showed in \cite{Aldaz2} that for the $\ell^\infty$ metric the constants $c_{m_n,1}$ grow to infinity as $n\rightarrow\infty$. As shown by Aubrun \cite{Aubrun}, in this case one has the estimate $c_{m_n,1}\geq C_\eps (\log n)^{1-\eps}$ for each $\eps>0$. It is still unknown if this also happens for other metrics, although there are partial results in the form of upper bounds for the possible growth of the constants. When considering the Euclidean metric a special argument allows to prove that $c_{m_n,1}=\mathcal O(n)$. For metrics given by general norms the best upper bound remains $c_{m_n,1}= \mathcal O(n\log n)$. Both results are due to E.M. Stein and J.O. Str\"omberg in \cite{SteinStromberg}. In this last paper, instead of the usual doubling condition, they used the fact that in $\real^n$ one has
\[
m_n(B(x,(1+1/n)R)) = (1+1/n)^n\ m_n(B(x,R)) \leq e\  m_n(B(x,R)).
\]
That is, the dilation by the factor $(1+1/n)$ does not increase essentially the volume of a ball. This allows to perform a more efficient covering argument with less overlapping than in Vitali's lemma.

\bigskip

A. Naor and T. Tao observed in \cite{NaorTao} that in order to extend Stein and Str\"om\-berg `$n\log n$' bound to general metric measure spaces suffices to assume that dilations by the factor $(1+1/n)$ preserve essentially the volumes of balls and that the measures of intersecting balls with the same radius are comparable. We next explain this conditions and state their result in detail.

\bigskip

Let $(X,d,\mu)$ be a metric measure space. Given a positive number $n\geq 1$, we will say that $\mu$ is $n$-micro-doubling if there exists a constant $K_0>0$ such that for each $x\in X$ and $R>0$ one has
\[
\mu(B((x,(1+1/n )R))\leq K_0\ \mu(B(x,R)).
\]
Of course $n$ here is not necessarily related to any notion of dimension. We will refer to $K_0$ as the $n$-micro-doubling constant.

\bigskip

We will say that a measure $\mu$ is weakly-doubling if the measure of two intersecting balls with the same radius is comparable. That is, there exists a constant $K_1>0$ such that for each $x,y \in X$ and $R>0$ such that $B(x,R)\cap B(y,R)\neq\emptyset$ one has
\[
\mu(B(y,R))\leq K_1\ \mu(B(x,R)).
\]

\bigskip

Following the terminology of A. Naor and T. Tao in \cite{NaorTao}, we will say that a measure $\mu$ is strong $n$-micro-doubling if it is both, $n$-micro-doubling and weakly-doubling. Equivalently, $\mu$ is strong $n$-micro-doubling if there exist a constant $K>0$ so that for each $x\in X$, $R>0$, and $y\in B(x,R)$ one has
\[
\mu(B(y,(1+1/n)R))\leq K\ \mu(B(x,R)).
\]

\bigskip

It is easy to see that every Ahlfors-David $n$-regular metric measure space is strong $n$-micro-doubling. We recall that a metric measure space $(X,d,\mu)$ is said to be Ahlfors-David $n$-regular if there exists a constant $C\geq 1$ so that
\[
\frac1{C}\ R^n \leq \mu(B(x,R)\leq C\ R^n,
\]
for all $x\in X$ and $R>0$. In the Euclidean $n$-dimensional case,  some examples of strong $n$-micro-doubling measures are given by the power densities $|x|^\alpha$, with $\alpha>-n$. See Section \ref{sect.medidas.microdoblantes} for more details.

\bigskip

We will consider maximal operators associated with the following type of time sequences in $(0,\infty)$. A sequence $\{a_k\}_{k\in \mathbb Z}\subset (0,\infty)$ is lacunary if there exist $a>1$ so that $a_{k+1}>a a_k$. If $a=n>1$ in this setting of $n$-micro-doubling measures, we will say explicitly that the time sequence is $n$-lacunary. Now we are ready to state the maximal theorem presented by A. Naor and T. Tao in \cite{NaorTao}.

\bigskip

\begin{teo}\label{teo.naor.tao} Let $(X,d,\mu)$ be a measure metric space, with $\mu$ a strong $n$-micro-doubling measure with constant $K$. Then we have the following weak type estimates
\begin{enumerate}[i)]
\item if\, $T\subset (0,\infty)$ is an $n$-lacunary sequence, then \footnote{Statement $i)$ of the theorem does not appear explicitly in \cite{NaorTao}.}
\[
\mu(\{x\in X: M_T f(x)>\lambda\}) \leq \frac{C_K}\lambda \|f\|_{L^1(\mu)},
\]
\item if\, $T\subset(0,\infty)$ is a lacunary sequence with constant $a$, then
\[
\mu(\{x\in X: M_T f(x)>\lambda\}) \leq \frac{C_K}\lambda \frac{\log n}{\log a} \|f\|_{L^1(\mu)},
\]
\item if\, $T=(0,\infty)$ so that $M=M_T$ is the Hardy-Littlewood maximal operator, then
\[
\mu(\{x\in X: Mf(x)>\lambda\}) \leq \frac{C_K n\log n}\lambda \|f\|_{L^1(\mu)},
\]
\end{enumerate}
where the $C_K$ are constants only depending on $K$.
\end{teo}

\bigskip

Note that part $iii)$ is a generalization of the `$n\log n$' bound by E.M. Stein and J.O. Str\"omberg (see \cite{SteinStromberg}) in normed Euclidean spaces. In this setting, parts $i)$ and $ii)$ are due to M.T. Men\'arguez and the second author in \cite{MenarguezSoria2}, as well as the chain of implications $i)\Rightarrow ii)\Rightarrow iii)$.

\bigskip

Moreover, A. Naor and T. Tao showed also in \cite{NaorTao} that the $\mathcal O(n\log n)$ bound is optimal even in the setting of Ahlfors-David $n$-regular spaces, by constructing a sequence of such spaces $(X_n,d_n,\mu_n)$ so that
\[
\|M_{\mu_n}\|_{L^1(\mu_n)\rightarrow L^{1,\infty}(\mu_n)}\geq Cn\log n.
\]

\bigskip

Theorem \ref{teo.naor.tao} can be proved as in \cite{MenarguezSoria2} or by an argument of Lindenstrauss present in \cite{Lindenstrauss} (see also \cite{NaorTao}). However, the main contribution of A. Naor and T. Tao in \cite{NaorTao} is the following nice localizaton principle for maximal operators, from which they obtained Theorem \ref{teo.naor.tao} as a corollary.

\bigskip

\begin{teo}\label{teo.localizacion} Let $(X,d,\mu)$ be a metric measure space and let $\mu$ be $n$-micro-doubling with constant $K_0$. If \ $T\subset (0,\infty)$ and $p\geq 1$, then we have the following localizaton property for the weak $L^p$ norms of the associated maximal operator
\[
\|M_T\|_{L^p\rightarrow L^{p,\infty}} \leq C_{K_0}+ C_{K_0}' \sup_{k\in\mathbb Z} \|M_k\|_{L^p\rightarrow L^{p,\infty}}.
\]
Here $M_k$ denotes the restriction of $M_T$ to times in $[n^k,n^{k+1}]$, that is $M_k = M_{T\cap[n^k,n^{k+1}]}$, and $C_{K_0}$ and $C_{K_0}'$ are constants only depending on $K_0$.\footnote{The constants $C_{K_0}$ and $C_{K_0}'$ obtained in \cite{NaorTao} are of the order of $\mathcal O(K_0)$ and $\mathcal O(1+\log\log K_0/(1+\log n))$ respectively.}
\end{teo}

\bigskip

The proof given by A. Naor and T. Tao of this result is probabilistic and relies on random martingales and Doob-type maximal inequalities. Here we present (see Theorem \ref{teo.localizacion.nuevo}) a more geometrical proof based on covering lemmas and selection processes that is closer in spirit to the arguments in \cite{SteinStromberg}. In addition, our result considers localization not only in time but also in space (see the statement of the theorem). The proof of this and the connection with Theorem \ref{teo.naor.tao} will be presented in Section \ref{seccion.localizacion}.

\bigskip

In the rest of the paper we restrict ourselves to the Euclidean case. We study measures given by a radial density, so that they can be defined in all dimensions. It is known that if these measures are finite, and hence are neither weakly-doubling nor $n$-micro-doubling, the (weak) $L^p$ operator norms of the associated maximal functions grow exponentially to infinity with the dimension at least for a range of values of $p$ near 1 (see \cite{Aldaz}, \cite{AldazPerezLazaro} and \cite{Criado}). There are cases in which this range of $p$'s may consist of the whole interval $[1,\infty)$. An example of this is given by the Gaussian measure (see \cite{CriadoSjogren}).

\bigskip

The situation changes if we consider measures that are strong $n$-micro-doub\-ling in each $\real^n$ with constants uniformly bounded in dimension. In this case Theorem \ref{teo.naor.tao} applies and the weak $L^1$ operator norms of the maximal function grow at most like $\mathcal O(n\log n)$. In Section \ref{sect.medidas.microdoblantes} we study maximal operators associated with uniformly weakly doubling measures. For such measures $\mu$ the main result of this section asserts that $M_\mu$ is uniformly bounded on $L^p(\mu,\real^n)$ for all $p>1$. The same thing is true in weak $L^1$ but restricting the action of $M_\mu$ to radial functions.

\bigskip

The previous results are obtained via the following characterization: $\mu$ is uniformly weakly doubling if and only if its density is essentially constant over dyadic annuli. This equivalence is proved using a new method of differentiation through dimensions presented in Section \ref{sect.diferenciacion}.

\bigskip

The case of decreasing densities is treated in Section \ref{sect.decrecientes}. The situation here is that the uniform weakly doubling property can be obtained from surprisingly mild conditions.

\bigskip

Finally, Section \ref{sect.further} contains examples of doubling measures with associated maximal operators failing to have uniform bounds.

\bigskip

\section{Localization principles}\label{seccion.localizacion}

\bigskip

In this section we will obtain Theorems \ref{teo.naor.tao} and \ref{teo.localizacion} as a consequence of yet another localization theorem, whose proof avoids technical arguments from probability theory. The statement is the following:

\bigskip

\begin{teo}\label{teo.localizacion.nuevo}
Let $(X,d,\mu)$ be a metric measure space and let $\mu$ be $n$-micro-doubling with constant $K_0$. If \ $T\subset (0,\infty)$ and $p\geq 1$ are fixed, then for each locally integrable function $f$ over $X$ and each $\lambda>0$ one can find a disjoint collection of measurable sets $\{A_k\}_{k\in \mathbb Z}$ such that
\begin{eqnarray*}
\mu(\{x\in X\!\!\!\!\!&:&\!\!\!\! \lambda<M_Tf(x)\leq2\lambda\}) \\&\leq& C_1 \left(\frac{1}{\lambda^p}\|f\|_{L^p(\mu)}^p + \sum_{k\in\mathbb Z}\mu(\{x\in X: M_k(f\chi_{A_k})(x)>C_2\lambda\})\right),
\end{eqnarray*}
where $C_1$ and $C_2$ are constants that only depend on $K_0$.
\end{teo}

\bigskip

It is not difficult to show that  Theorem \ref{teo.localizacion.nuevo} implies Theorem \ref{teo.localizacion}. To see this, given a locally integrable $f$ and $\lambda>0$ we write
\begin{eqnarray*}
E_\lambda&=&\{x\in X: M_Tf(x)>\lambda\},\\
F_\lambda&=&\{x\in X: \lambda<M_Tf(x)\leq2\lambda\}.
\end{eqnarray*}
Note that by the disjointness of the collection $\{A_k\}_{k\in \mathbb Z}$ we have
\begin{eqnarray*}
\sum_{k\in\mathbb Z}\mu(\{x\in X: M_k(f\chi_{A_k})(x)>C_2\lambda\}) &\leq& \frac{C_2^{-p}}{\lambda^p} \sum_{k\in\mathbb Z} \|M_k\|_{L^p\rightarrow L^{p,\infty}}^p \int_{A_k}|f|^p\,d\mu \\&\leq& \frac{C_2^{-p}}{\lambda^p} \sup_{k\in\mathbb Z} \|M_k\|_{L^p\rightarrow L^{p,\infty}}^p \int_X|f|^p\,d\mu.
\end{eqnarray*}
Hence Theorem \ref{teo.localizacion.nuevo} gives the estimate
\[
\mu(F_\lambda) \leq \frac{C_{K_0}}{\lambda^p} \left(1+\sup_{k\in\mathbb Z} \|M_k\|_{L^p\rightarrow L^{p,\infty}}^p\right)  \|f\|_{L^p(\mu)}^p.
\]
This implies
\begin{eqnarray*}
\mu(E_\lambda)&=&\mu\left(\bigcup_{j=0}^\infty F_{2^j\lambda}\right) = \sum_{j=0}^\infty \mu(F_{2^j\lambda})\\&\leq&  \sum_{j=0}^\infty \frac{C_{K_0}}{(2^j\lambda)^p} \left(1+\sup_{k\in\mathbb Z} \|M_k\|_{L^p\rightarrow L^{p,\infty}}^p\right)  \|f\|_{L^p(\mu)}^p \\&=& \frac{C_{K_0}}{\lambda^p} \left(1+\sup_{k\in\mathbb Z} \|M_k\|_{L^p\rightarrow L^{p,\infty}}^p\right)  \|f\|_{L^p(\mu)}^p,
\end{eqnarray*}
as wanted.

\bigskip

Also part $i)$ of Theorem \ref{teo.naor.tao} follows from Theorem \ref{teo.localizacion} using the following argument.

\bigskip

If $T=\{a_j\}_{j\in\mathbb Z}$ is an $n$-lacunary sequence then each interval of the form $[n^k,n^{k+1}]$ contains at most one element of $T$. Theorem \ref{teo.localizacion} implies that
\[
\|M_T\|_{L^1\rightarrow L^{1,\infty}} \leq C_K \sup_{j\in \mathbb Z} \|M_{\{a_j\}}\|_{L^1\rightarrow L^{1,\infty}}.
\]
We only have to show now that $M_{\{a_j\}}$ is bounded independently of $j$. Since
\[
M_{\{a_j\}}f(x)=\frac{1}{\mu(B(x,a_j))}\int_{B(x,a_j)} |f(y)|\,d\mu(y),
\]
by the weak doubling property of $\mu$ and Fubini Theorem we have
\begin{eqnarray*}
\|M_{\{a_j\}}f\|_{L^1} &=& \int_X \frac{1}{\mu(B(x,a_j))}\int_X \chi_{B(x,a_j)}(y) |f(y)|\,d\mu(y)\,d\mu(x)\\ &\leq& \int_X \int_X \frac{K_1}{\mu(B(y,a_j))}\chi_{B(y,a_j)}(x) |f(y)|\,d\mu(y)\,d\mu(x) \\ &=& \int_X \frac{K_1}{\mu(B(y,a_j))}\int_{B(y,a_j)}\,d\mu(x)\,|f(y)|\,d\mu(y) = K_1 \|f\|_{L^1}.
\end{eqnarray*}
This says, in particular, that each $M_{\{a_j\}}$ is bounded on $L^1$ with operator norm bounded by $K_1$.

\bigskip

The implications $i)\Rightarrow ii)$ and $ii)\Rightarrow iii)$ of Theorem \ref{teo.naor.tao} can be obtained through the same generic arguments given in \cite{MenarguezSoria2}. The first one says that every maximal operator associated with a lacunary sequence is essentially majorized by $\log n$ operators associated with $n$-lacunary sequences. The second one says that the full maximal operator is majorized by $n$ operators, each associated with a lacunary sequence.

\bigskip

We now give a proof of Theorem \ref{teo.localizacion.nuevo}.

\bigskip

\begin{proof}[Proof of Theorem \ref{teo.localizacion.nuevo}]
Given a locally integrable $f$ and $\lambda>0$ we write
\[
F_\lambda=\{x\in X: \lambda<M_Tf(x)\leq2\lambda\}.
\]
Once we determine what the sets $A_j$ are, we will have to prove that
\[
\mu(F_\lambda)\leq C_1 V_{\lambda,p}^f,
\]
where we have used the notation
\[
V_{\lambda,p}^f=\frac{1}{\lambda^p}\|f\|_{L^p(\mu)}^p + \sum_{k\in\mathbb Z}\mu(\{x\in X: M_k(f\chi_{A_k})(x)>C_2\lambda\}).
\]

\bigskip

Note first that it is enough to prove the result for $p=1$. To see this, given a function $f$ in $L^p(\mu)$ we consider $f=f^\lambda+f_\lambda$ where
\begin{eqnarray*}
f^\lambda&=&f\chi_{\{|f|>\lambda\}} \in L^1(\mu),\\
f_\lambda&=&f\chi_{\{|f|\leq \lambda\}} \in L^\infty(\mu).
\end{eqnarray*}
Since
\[
M_T f(x)\leq M_T f^\lambda(x)+M_T f_\lambda(x) \leq M_T f^\lambda(x)+\lambda,
\]
for $M_Tf(x)>2\lambda$ to hold it is necessary that $M_T f^\lambda(x) >\lambda$. Hence,
\[
F_{2\lambda}\subset \{x\in X: \lambda<M_Tf^\lambda(x)\leq 4\lambda\} =
G_1\cup G_2,
\]
with
\begin{eqnarray*}
G_1&=& \{x\in X: \lambda<M_Tf^\lambda(x)\leq 2\lambda\},\\ G_2&=&\{x\in X: 2\lambda<M_Tf^\lambda(x)\leq 4\lambda\}. \end{eqnarray*}
Then since $f^\lambda\in L^1(\mu)$ we apply the result for $p=1$ to $G_1$ and $G_2$ to obtain
\begin{eqnarray*}
\mu(F_{2\lambda}) &\leq& 2\max(\mu(G_1),\mu(G_2))\\ &\leq&C \left( C_1 \int_X  \frac{|f^\lambda|}{\lambda}\,d\mu + \sum_{k\in\mathbb Z}\mu(\{x\in X: M_k(f^\lambda\chi_{A_k})(x)>C_2\lambda\})\right) \\ &\leq&C \left( C_1 \int_X  \frac{|f|^p}{\lambda^p}\,d\mu + \sum_{k\in\mathbb Z}\mu(\{x\in X: M_k(f\chi_{A_k})(x)>C_2\lambda\})\right),
\end{eqnarray*}
for each $p\geq 1$. Since $\lambda$ is arbitrary the result is proved for any $p\geq 1$.

\bigskip

Now we prove the result in the case that $p=1$, that is
\[
\mu(F_\lambda) \leq C_1 V_{\lambda,1}^f
\]
 We consider the following collections of balls
\begin{eqnarray*}
\mathcal B&=&\{ B(x,R): x\in X,\,R\in T\},\\
\mathcal B_k&=& \{B(x,R):x\in X,\,R\in T\cap [n^k,n^{k+1}]\}.
\end{eqnarray*}
Given a ball $B\in \mathcal B$ we denote by $z_B$ its center and by $R_B$ its radius. We define the collection
\[
\mathcal A =\left\{B\in\mathcal B: \lambda<\frac{1}{\mu(B)}\int_B |f|\,d\mu\leq 2\lambda\right\}.
\]
For each $B\in\mathcal A$ one can find a concentric ball $\tilde B$ so that $R_B/(1+1/n)\leq R_{\tilde B} < R_B$ and
\[
\frac{1}{\mu(B)}\int_{\tilde B} |f|\,d\mu > \lambda.
\]
We will write $B^0=B(z_B,R_B-R_{\tilde B})$. Since
\[
F_\lambda \subset \bigcup_{B\in \mathcal A} B^0,
\]
it suffices to prove that
\[
\mu\left(\bigcup_{B\in \mathcal A} B^0\right) \leq C_1 V_{\lambda,1}^f.
\]
It is not difficult to see (see the argument preceding (\ref{bolas.en.nivel}) below) that $\cup_{B\in\mathcal A} B^0$ is contained in a level set of $M_Tf$ and, as a consequence, has finite $\mu$-measure. Using that $X$ is separable and the monotone convergence there exists $\mathcal A'\subset \mathcal A$ with $\sharp \mathcal A'<\infty$ such that
\[
\mu\left(\bigcup_{B\in \mathcal A'} B^0\right) \geq \frac12 \mu\left(\bigcup_{B\in \mathcal A} B^0\right).
\]
Hence, it is enough to show that
\[
\mu\left(\bigcup_{B\in \mathcal A'} B^0\right) \leq C_1 V_{\lambda,1}^f,
\]
with $C$ independent of $\mathcal A'$. Writing
\[
\mathcal A_k:=\{B\in \mathcal A': B\in \mathcal B_k\},
\]
we have that $\mathcal A' = \mathcal A^1 \cup \mathcal A^2$ with
\begin{eqnarray*}
\mathcal A^1&=& \bigcup_{k \mbox{\scriptsize{ odd}}} \mathcal A_k,\\
\mathcal A^2&=& \bigcup_{k \mbox{\scriptsize{ even}}} \mathcal A_k.
\end{eqnarray*}
Assume that $\mu(\bigcup_{B\in\mathcal A^1}B^0)\geq \mu(\bigcup_{B\in\mathcal A^2}B^0)$ (if not, interchange the names), then $\mu(\bigcup_{B\in\mathcal A^1}B^0)\geq \frac12\mu(\bigcup_{B\in\mathcal A'}B^0)$ and we just need to prove that
\[
\mu\left(\bigcup_{B\in\mathcal A^1}B^0\right) \leq C_1 V_{\lambda,1}^f.
\]
For the sake of simplicity in the notation we rename $\mathcal A=\mathcal A^1$.

\bigskip

Now we want to write the set $\bigcup_{B\in\mathcal A}B^0$ as a union of disjoint sets. Observe that $\mathcal A = \{B_j\}_{j=1,\cdots, J}$ for certain $J$. We define $D_{B_1}=B_1^0$, and if $D_{B_1},\cdots, D_{B_m}$ are already defined, we take
\[
D_{B_{m+1}}=B_{m+1}\setminus\bigcup_{j=1}^m B_j^0.
\]
Then, we have
\[
\mu\left(\bigcup_{B\in\mathcal A} B^0\right) = \mu\left(\bigcup_{B\in\mathcal A} D_B\right).
\]
We also define the functions
\[
g_B(x)=\frac{\mu(D_B)}{\mu(B)} \chi_{\tilde B}(x)
\]
for $B\in\mathcal A$ and
\[
G_k(x)=\sum_{B\in\mathcal A_k} g_B(x).
\]

\bigskip

We start a selection process. Take $k_1$ as the largest $k\in \mathbb Z$ with $\mathcal A_k\neq \emptyset$, then we take $\tilde G_{k_1} = G_{k_1}$ and $\tilde{\mathcal A}_{k_1} = \mathcal A_{k_1}$. Once $\tilde G_{k_1},\cdots,\tilde G_{k_{m-1}}$ are determined for $m\geq 2$, we take $k_m$ as the largest $k< k_{m-1}$ so that $\mathcal A_k \neq \emptyset$. We say that $B\in\tilde{\mathcal A}_{k_m}$ if $B\in\mathcal A_{k_{m}}$ and
\[
\sum_{j=1}^{m-1} \tilde G_{k_j}\leq 1
\]
on $B$. Then we define
\[
\tilde G_{k_{m}}(x) = \sum_{B\in\tilde{\mathcal A}_{k_m}} g_B(x).
\]
Since $\mathcal A$ is finite this process ends in a finite number of steps, and we have obtained $\tilde G_{k_1},\cdots,\tilde G_{k_M}$ for certain $M$. We call
\[
\tilde{\mathcal A}=\bigcup_{m=1}^M \tilde{\mathcal A}_{k_m},
\]
and claim that
\begin{equation}\label{claim.reduccion}
\mu\left(\bigcup_{B\in\mathcal A} D_B\right) \leq (1+K_0) \sum_{B\in\tilde{\mathcal A}} \mu(D_B).
\end{equation}

\bigskip

To prove this claim note that if $A\in \mathcal A_{k_m}\setminus \tilde{\mathcal A}_{k_m}$, then there exists $z\in A$ such that
\[
\sum_{j=1}^{m-1} \tilde G_{k_j}(z)>1.
\]
Note that if $A\cap \tilde B\neq \emptyset$ for some $B\in \tilde{\mathcal A}_{k_j}$ with $j<m$, then $A\subset B^\ast:=B(x_B,(1+1/n)R_B)$. Thus for all $z\in A$ we have
\[
\sum_{j=1}^{m-1} G_{k_j}^\ast(z) := \sum_{j=1}^{m-1} \sum_{B\in\tilde{\mathcal A}_{k_j}}\frac{\mu(D_B)}{\mu(B)}\chi_{B^\ast}(z)> 1.
\]
Then, by Tchebychev inequality

\begin{eqnarray*}
\mu\left(\bigcup_{B\in\mathcal A\setminus\tilde{\mathcal A}} D_B\right) &\leq& \mu\left(\Bigl\{z\in X: \sum_{j=1}^M G_{k_j}^\ast(z)>1\Bigr\}\right) \\&\leq& \sum_{j=1}^M \int_X G_{k_j}^\ast(z)\,d\mu(z)= \sum_{B\in\tilde{\mathcal A}}^M \frac{\mu(D_B)}{\mu(B)}\mu(B^\ast) \\&\leq& K_0 \sum_{B\in \tilde{\mathcal A}} \mu(D_B).
\end{eqnarray*}

\bigskip

Now that the claim is justified, we only need to prove
\begin{equation}\label{desigualdad.lados}
\sum_{B\in \tilde{\mathcal A}} \mu(D_B) \leq C_1 V_{\lambda,1}^f.
\end{equation}
By the definition of $\tilde B$ we have
\begin{equation}\label{suma.g}
\sum_{B\in \tilde{\mathcal A}} \mu(D_B) \leq \frac1\lambda  \sum_{B\in \tilde{\mathcal A}} \frac1{\mu(B)}\int_{\tilde B}|f|\,d\mu\ \mu(D_B) = \frac1\lambda \int_X |f| \left(\sum_{j=1}^M \tilde G_{k_j}\right)\,d\mu.
\end{equation}

\bigskip

For each $j=1,\cdots ,M$ we define $\tilde A_j:= \supp \tilde G_{k_j}$. We now take $A_M=\tilde A_M$ and
\[
A_j=\tilde A_j\setminus \bigcup_{\ell=j+1}^{M}\tilde A_\ell.
\]
If $z\in A_m$ we have $\tilde G_{k_j}(z)=0$ if $j>m$, and then, by the way we selected $\tilde G_{k_m}$
\[
\sum_{j=1}^M \tilde G_{k_j}(z) = \tilde G_{k_m}(z)+ \sum_{j=1}^{m-1} \tilde G_{k_j}(z) \leq \tilde G_{k_m}(z) + 1.
\]
Then we have that
\[
\sum_{j=1}^M \tilde G_{k_j}\leq \sum_{j=1}^M \tilde G_{k_j}\chi_{A_j}+1,
\]
which combined with (\ref{claim.reduccion}) and (\ref{suma.g}) yields
\begin{equation*}\label{cotas.dos.lados}
\sum_{B\in \tilde{\mathcal A}} \mu(D_B) \leq \frac{1}\lambda\left(\int_X |f|\,d\mu+\sum_{j=1}^M \int_X |f| \chi_{A_j} \tilde G_{k_j}\,d\mu\right).
\end{equation*}
In order to bound the last sum note that
\[
\sum_{j=1}^M \int_X f \chi_{A_j} \tilde G_{k_j}\,d\mu = \sum_{j=1}^M \sum_{B\in\tilde{\mathcal A}_{k_j}} \frac1{\mu(B)}\int_{\tilde B} |f|\chi_{A_j}\,d\mu\ \mu(D_B)
\]
Given $B\in \tilde{\mathcal A}_{k_j}$, we say that $B\in \tilde{\mathcal A}_{k_j}^\ast$ if
\[
\frac1{\mu(B)}\int_{\tilde B} |f|\chi_{A_j}\,d\mu \leq \frac\lambda{2}.
\]
Hence, we have
\[
\frac{1}\lambda\sum_{j=1}^M \sum_{B\in\tilde{\mathcal A}_{k_j}^\ast} \frac1{\mu(B)}\int_{\tilde B} |f|\chi_{A_j}\,d\mu\ \mu(D_B) \leq \frac12 \sum_{B\in \tilde{\mathcal  A}} \mu(D_B)
\]
and this term can be absorbed in the left hand side of (\ref{desigualdad.lados}).

\bigskip

We now consider those balls that do not belong to any of the classes $\tilde{\mathcal A}_{k}^\ast$.
We claim that if $z\in B^0$, then $\tilde B \subset B(z,R_B)\subset B^\ast$. The $n$-micro-doubling condition would imply then that $\mu(B)\geq \mu(B^\ast)/K_0 \geq \mu(B(z,R_B))/K_0$ and consequently if $B\in \tilde{\mathcal A}_{k_j}\setminus \tilde{\mathcal A}_{k_j}^\ast$
\begin{eqnarray*}
\frac{\lambda}{2}&<&\frac1{\mu(B)}\int_{\tilde B} |f|\chi_{A_j}\,d\mu \leq \frac {K_0}{\mu(B(z,R_B))}\int_{B(z,R_B)} |f|\chi_{A_j}\,d\mu \\ &\leq& K_0 M_{k_j} (f\chi_{A_j})(z).
\end{eqnarray*}
Thus, for such $B$,
\begin{equation}\label{bolas.en.nivel}
B^0\subset \left\{x\in X: M_{k_j} (f\chi_{A_j})(x)>\frac\lambda{2K_0}\right\}.
\end{equation}
Therefore, using that for $B\in\mathcal A$ one has
\begin{equation*}
\frac1{\mu(B)}\int_{\tilde B} |f|\chi_{A_j}\,d\mu \leq \frac1{\mu(B)}\int_{B} |f|\chi_{A_j}\,d\mu \leq 2\lambda,
\end{equation*}
we conclude that
\begin{eqnarray*}
\frac{1}\lambda \sum_{j=1}^M \sum_{B\in\tilde{\mathcal A}_{k_j}\setminus \tilde{\mathcal A}_{k_j}^\ast} \!\!\!\!\!\!&\!\!\!\!\!\!\!\!&\!\!\!\!\!\!\!\!\frac1{\mu(B)}\int_{\tilde B} |f|\chi_{A_j}\,d\mu\ \mu(D_B) \\&\leq& 2 \sum_{j=1}^M \sum_{B\in\tilde{\mathcal A}_{k_j}\setminus \tilde{\mathcal A}_{k_j}^\ast} \mu(D_B)\\&\leq& 2
\sum_{j=1}^M \mu\left(\left\{x\in X: M_{k_j} (f\chi_{A_j})(x)>\frac\lambda{2K_0}\right\}\right).
\end{eqnarray*}
This finishes the proof, provided we justify the last claim.

\bigskip

In order to do so, suppose that $y\in\tilde B$, then
\[
|y-z|\leq |y-z_B|+|z_B-z| < R_{\tilde B} + (R_B-R_{\tilde B}) = R_B,
\]
which means that $y\in B(z,R_B)$. Assume now that $y\in B(z,R_B)$
\[
|y-z_B|\leq |y-z|+|z-z_B| < R_B + R_B/n = (1+1/n)R_B,
\]
hence $y\in B^\ast$. The claim is proved.
\end{proof}

\bigskip

\section{Uniform bounds for maximal functions associated with rotation invariant measures in Euclidean spaces}\label{sect.euclidean.case}

\bigskip

In the sequel we will consider $X$ to be the Euclidean space $\real^n$, equipped with the Euclidean distance and a rotation invariant measure with a radial density, so that it can be defined in all dimensions. The goal of this section is to show that, in this setting, uniformly weakly doubling measures have associated maximal operators satisfying dimension free bounds on the spaces $L^p$.

\bigskip

The notation that we will use is the following. By $\mu$ we will always denote a rotation invariant measure with a radial density $w$. That is $d\mu(x)=w(x)\,dx$ and $w(x)=w_0(|x|)$ with $w_0:[0,\infty)\rightarrow[0,\infty)$. Properly speaking, $w$ and $d\mu$ are objects that change with the dimension. In particular, their integrability properties may be different from one Euclidean space to another. By a slight abuse of notation, we keep, however the same symbols to denote them in $\real^n$ for all $n$.

\bigskip

As usual, for a measurable set $E\subset\real^n$, $|E|$ will denote its Lebesgue measure, $d\sigma_{n-1}$ will be the measure on $\mathbb S^{n-1}$ induced by Lebesgue measure on $\real^n$ and $\omega_{n-1}=\sigma_{n-1}(\mathbb S^{n-1})$. Also, for each $x\in\real^n$ and $R>0$, we will denote by $B(x,R)$ the Euclidean ball centered at $x$ with radius $R$. For a ball whose center is the origin we will use the notation $B_R$.

\bigskip

\subsection{Weakly-doubling measures in Euclidean spaces and uniform bounds}\label{sect.medidas.microdoblantes}

\bigskip

As a first and simple example that uniformity in the weak doubling condition implies uniform bounds we present

\bigskip

\begin{teo}\label{teorema.radiales} Let $\mu$ be a rotation invariant measure over $\real^n$ that is absolutely continuous with respect to Lebesgue measure. If $\mu$ is weakly doubling with constant $K_1$ and $f$ is a radial function we have
\[
\mu(\{x\in \real^n: M_\mu f(x)>\lambda\})\leq \frac{2(K_1+1)}\lambda \|f\|_{L^1(\real^n,\mu)}.
\]
\end{teo}

\bigskip

Observe that if $K_1$ is uniformly bounded in $n$, the above inequality provides a dimension-free bound for $M_\mu$. This is a generalization of a result presented by M.T. Men\'arguez and the second author in \cite{MenarguezSoria2}. There, it was shown that for all radial $f$ over $\real^n$ and $\lambda>0$ one has
\begin{equation}\label{teo.orejas}
|\{x\in\real^n:Mf(x)>\lambda\}| \leq \frac 4\lambda\|f\|_{L^1(\real^n,dx)}.
\end{equation}
A previous generalization was due to A. Infante, who pointed out (see \cite{AdrianInfante}, \cite{Infante}) that the same is true for $M_\mu$, if $\mu$ is an `increasing' measure. By this we mean that $\mu(B(x,R))\geq \mu(B(y,R))$ whenever $y=a x$ with $a\geq 1$. In \cite{APL} there is an independent approach with results similar to our Theorems \ref{teorema.radiales} and \ref{teo.alphas}. These two results already appeared in \cite{tesis}.

\bigskip

The main result of the section asserts that the uniform weakly-doubling property of a measure $\mu$ ensures uniform bounds in the dimension not only for radial functions but, in fact, over the whole class of functions $L^p(\real^n,d\mu)$, if $p>1$.

\bigskip

\begin{teo}\label{teo.wdoublin.unif} If there is a positive integer $N$ so that $\mu$ is uniformly weakly-doubling in $\real^n$ for  $n\geq N$, then for each $p>1$ there exists $C_p$ only depending on $\mu$ and $p$ so that
\[
\|M_\mu f\|_{L^p(\real^n,\mu)} \leq C_p \|f\|_{L^p(\real^n,\mu)},
\]
for all $f\in L^p(\real^n,\mu)$.
\end{teo}

\bigskip

In order to prove this result we first obtain the following characterization of radial measures that are uniformly weakly-doubling.

\bigskip

\begin{teo}\label{teo.ecoda.wdoubling} There exists $N\in\mathbb N$ so that $d\mu=w(x)dx$, with $w(x)=w_0(|x|)$, is uniformly weakly-doubling in $\real^n$ for each $n\geq N$ if and only if $w$ is essentially constant on dyadic annuli, i.e. there exists $\beta\geq 1$ so that for all $R>0$ one has
\begin{equation}\label{ecoda}
\esssup_{R\leq |x|\leq 2R} w(x) \leq \beta \essinf_{R\leq |x|\leq 2R} w(x).
\end{equation}
\end{teo}

\bigskip

The proof of this employs a new method of differentiation through dimensions that is presented in Section \ref{sect.diferenciacion}. With this characterization at hand, instead of proving Theorem \ref{teo.wdoublin.unif} we will prove the equivalent:

\bigskip

\begin{teo}\label{teo.ecoda.unif} With the above notation, let $w$ be essentially constant over dyadic annuli with constant $\beta$. Then, there exists $N$ so that $\mu$ is locally finite in $\real^n$ for $n\geq N$ and, for each $p>1$, there exists a constant $C_p$ only depending on $p$ and $\beta$ so that for  $n\geq N$ one has
\[
\|M_\mu f\|_{L^p(\real^n,\mu)} \leq C_p \| f\|_{L^p(\real^n,\mu)},
\]
for all $f\in L^p(\real^n,\mu)$.
\end{teo}

\bigskip

The densities $|x|^\alpha$, $\alpha\in\real$, provide examples of measures with this property. It is easy to see that in each case they are essentially constant over dyadic annuli with constant $2^{|\alpha|}$.

\bigskip

At this point, let us make the simple observation that if $w$ is essentially constant over dyadic annuli, then $w$ can be compared pointwise with a continuous function in $\mathbb R^n\setminus \! \{0\}$. In the sequel, we will assume, as we may, that under this hypothesis, $w$ is indeed continuous with the possible exception at $x=0$.

\bigskip

Theorem \ref{teo.ecoda.unif} will be obtained as a consequence of the following results.

\bigskip

\begin{lema}\label{lema.acotacion.teo.power} Let $\mu$ be locally finite in $\real^n$ and assume that its density $w$ is essentially constant over dyadic intervals with constant $\beta$. Then we have the following pointwise inequality
\[
M_\mu f(x)\leq C\left(Mf(x)+\frac1{w(x)}M(fw)(x)+\mathcal H_\mu f(x)\right),
\]
where $C$ depends only on $\beta$, $M$ is the usual Hardy-Littlewood maximal operator  and $\mathcal H_\mu$ is defined as
\[
\mathcal H_\mu f(x)=\sup_{R\geq |x|} \frac1{\mu(B_R)}\int_{B_R}|f(y)|\,d\mu(y).
\]
\end{lema}

\bigskip

This transforms the problem of finding $L^p(\mu)$ bounds for $M_\mu$ to the one of finding them for $\mathcal{H}_\mu$ and for $M$.

\bigskip

\begin{teo}\label{teo.power.lp}
In the same hypotheses of Lemma \ref{lema.acotacion.teo.power}, assume that in addition we have the weighted inequalities
\begin{equation} \label{W}
\begin{array}{rcl}
\|Mf\|_{L^p(w)}&\leq & W_1 \|f\|_{L^p(w)},\\
\|Mf\|_{L^p(w^{1-p})}&\leq & W_2 \|f\|_{L^p(w^{1-p})}.
\end{array}
\end{equation}
Then one has
\[
\|M_\mu f\|_{L^p(\mu)}\leq C \|f\|_{L^p(\mu)},
\]
where the constant $C$ only depends on $\beta$, $W_1$ and $W_2$.
\end{teo}

\bigskip

Note that the required weighted inequalities are equivalent to $w \in A_p \cap A_{p'}=A_{\min\{p,p'\}}$, where $A_p=A_p(\real^n)$ denotes the usual Muckenhoupt class of weights. The existence of $W_1$ and $W_2$ in (\ref{W}) is guaranteed due to the following

\bigskip

\begin{lema}\label{lema.ecoda.ap}
Let $w$ be essentially constant over dyadic intervals with constant $\beta$. Then for each $p>1$ there exist $N\in\mathbb N$ depending only on $p$ and $\beta$ so that $w\in A_p(\real^n)$ for all $n\geq N$.
\end{lema}

\bigskip

The constants $W_1$ and $W_2$ can be taken, in fact, independent of the dimension, as the following result by J. Duoandikoetxea and L. Vega, appeared in \cite{DuoandiVega}, shows.

\bigskip

\begin{teo}\label{teo.duoandi.vega} Let $w_0$ be a nonnegative function on $(0,\infty)$, so that $w=w_0(|\cdot|)\in A_p(\real^N)$. Then for all $n\geq N$ one has $w\in A_p(\real^n)$. Moreover,
\[
\|Mf\|_{L^p(\real^n,w)} \leq C \|f\|_{L^p(\real^n,w)},
\]
with a constant $C$ that might depend on $p$ and $w$ but not on $n$.
\end{teo}

\bigskip

To complete the proof of Theorem \ref{teo.ecoda.unif}, we have to show that if $w$ is essentially constant over dyadic annuli, then there exists $N\in\mathbb N$ so that $w$ is locally integrable in $\real^n$ for $n>N$. This is immediate once we have the following auxiliary lemma, which will be used  in different proofs throughout the paper.

\bigskip

\begin{lema}\label{lema.ecoda.hardy} Let $w$ be essentially constant over dyadic intervals with constant $\beta$. Then there exists $N\in\mathbb N$ only depending on $\beta$, so that for $n\geq N$ one has in $\real^n$ the control
\begin{equation}\label{hardy.aveces}
\mu(B_R)=\int_{B_R}w(x)\,dx \leq 2\beta\ w_0(R)|B_R|.
\end{equation}
\end{lema}

\bigskip

In the remaining part of the section we present the proofs of Theorem \ref{teo.power.lp}, Lemmas \ref{lema.acotacion.teo.power}, \ref{lema.ecoda.ap} and \ref{lema.ecoda.hardy} and Theorem \ref{teorema.radiales} in this given order.

\bigskip

We start showing how Lemma \ref{lema.acotacion.teo.power} implies Theorem \ref{teo.power.lp}.

\bigskip

\begin{proof}[Proof of Theorem \ref{teo.power.lp}.]
By Lemma \ref{lema.acotacion.teo.power} one has
\[
\|M_\mu f\|_{L^p(\mu)} \leq C\left(\|Mf\|_{L^p(\mu)}+\left\|\frac{M(fw)}{w}\right\|_{L^p(\mu)} + \|\mathcal H_\mu f\|_{L^p(\mu)}\right),
\]
where the constant $C$  depends only on $\beta$.

\bigskip

By assumption we have
\[
\|Mf\|_{L^p(\mu)}\leq W_1 \|f\|_{L^p(\mu)}.
\]

\bigskip

We also would like to have
\[
\left\|\frac1{w}M(fw)\right\|_{L^p(\mu)} \leq C\|f\|_{L^p(\mu)}.
\]
Taking $g=fw$ this is equivalent to
\[
\|Mg\|_{L^p(w^{1-p})}\leq C \|g\|_{L^p(w^{1-p})},
\]
which we know true by assumption with $C=W_2$.

\bigskip

For $\mathcal H_\mu$ we use the standard argument for Hardy type operators. It is obvious that $\mathcal H_\mu$ is bounded on $L^\infty(\mu)$ with constant 1. We will show that it is also weakly bounded on $L^1(\mu)$ with operator norm 1. Then by real interpolation it is bounded on $L^p(\mu)$ with operator norm controlled by an absolute constant.

\bigskip

To see the weak type inequality take $\lambda>0$ and consider $E_\lambda=\{x\in\real^n: \mathcal H_\mu f(x)\geq \lambda\}$. If $x\in E_\lambda$ there exists $R_x>|x|$ so that
\[
\frac1{\mu(B_{R_x})}\int_{B_{R_x}} |f(y)|\,d\mu(y)\geq \lambda.
\]
Note that then $B_{R_x}\subset E_\lambda$, and that
\[
E_\lambda=\bigcup_{x\in E_\lambda} B_{R_x}.
\]
Then $E_\lambda = B_R$ for certain $R>0$ and monotonicity gives
\begin{eqnarray*}
\mu(E_\lambda)&=&\mu(B_R) = \sup_{x\in E_\lambda} \mu(B_{R_x}) \leq \sup_{x\in E_\lambda} \frac1\lambda \int_{B_{R_x}} |f(y)|\,d\mu(y) \\&\leq& \frac1\lambda \int_{\real^n} |f(y)|\,d\mu(y).
\end{eqnarray*}
\end{proof}

\bigskip

\begin{proof}[Proof of Lemma \ref{lema.acotacion.teo.power}.]
We will bound the mean value over an arbitrary ball $B(x,R)$. Fixing $x$ and $R$, we consider different cases.

\bigskip

If $|x|\geq 2R$ and $y\in B(x,R)$ then
\[
\frac12 |x|\leq |x|-R\leq |y|\leq |x|+R\leq \frac32|x|.
\]
Since $w$ is essentially constant over dyadic annuli, we have $\beta^{-1} w(x) \leq w(y)\leq \beta w(x)$. Hence
\begin{eqnarray*}
\frac1{\mu(B(x,R))}\int_{B(x,R)}|f(y)|\,w(y)\,dy  &\leq& \frac{\beta^2}{w(x)\,|B(x,R)|}\int_{B(x,R)} |f(y)|\,w(x)\,dy\\&\leq& \beta^2 Mf(x).
\end{eqnarray*}

\bigskip

In the case that $R/2\leq |x|\leq 2R$ one has that if $y\in B(x,R)\setminus B_{R/2}$, then $|x|/4 \leq |y| \leq 3|x|$, which implies $\beta^{-2}w(|x|)\leq w(y) \leq \beta^2 w(|x|)$. Hence,
\begin{eqnarray*}
\mu(B(x,R)) &\geq& \mu(B(x,R)\setminus B_{R/2}) \geq \frac1{\beta^{2}}\ w(x)\ |B(x,R)\setminus B_{R/2}| \\ &\geq& \frac1{2\beta^{2}}\ w(x)\ |B(x,R)|.
\end{eqnarray*}
Therefore we have
\begin{eqnarray*}
\frac1{\mu(B(x,R))}\int_{B(x,R)}|f(y)|\,w(y)\,dy&\leq&  \frac{2\beta^2}{w(x)|B(x,R)|} \int_{B(x,R)} |f(y)|\,w(y)\,dy \\&\leq& \frac{2\beta^2}{w(x)} M(fw)(x).
\end{eqnarray*}

\bigskip

Last, we consider that $|x|\leq R/2$.
We split the ball $B(x,R)$ into two disjoint pieces $B_{R/2}$ and $B(x,R)\setminus B_{R/2}$ and integrate over them separately. For the first one
\begin{equation*}
\frac1{\mu(B(x,R))}\int_{B_{R/2}} |f(y)|\,d\mu(y) \leq \frac{1}{\mu(B_{R/2})}\int_{B_{R/2}} |f(y)|\,d\mu(y) \leq \mathcal H_\mu f(x).
\end{equation*}
For the second one note that if $y\in B(x,R)\setminus B_{R/2}$ then $R/2\leq |y|\leq 3R$ and then $\beta^{-2} w_0(R)\leq w(y)\leq\beta^2 w_0(R)$. This implies that
\begin{eqnarray*}
\mu(B(x,R)) &\geq& \mu(B(x,R)\setminus B_{R/2}) \geq \frac1{\beta^{2}}\ w_0(R)\ |B(x,R)\setminus B_{R/2}| \\&\geq& \frac1{2\beta^{2}}\ w_0(R)\ |B(x,R)|,
\end{eqnarray*}
and then we get
\begin{eqnarray*}
\frac1{\mu(B(x,R))}&& \!\!\!\!\!\!\!\!\!\!\!\!\!\int_{B(x,R)\setminus B_{R/2}} |f(y)|\,d\mu(y) \\&\leq& \frac{2\beta^4}{w_0(R)\ |B(x,R)|} \int_{B(x,R)\setminus B_{R/2}} |f(y)|\,w_0(R)\,dy \\&\leq& 2\beta^4 Mf(x).
\end{eqnarray*}
\end{proof}

\bigskip

\begin{xrem} Both, Theorem \ref{teo.power.lp} and Lemma \ref{lema.acotacion.teo.power} do not require $\mu$ to be radial. For the applications, however, this requirement is the most natural in order to define $\mu$ and $M_\mu$ simultaneously in all dimensions and to study whether or not there are uniform bounds as $n\longrightarrow\infty$.
\end{xrem}

\bigskip

\begin{proof}[Proof of Lemma \ref{lema.ecoda.ap}]
We have to prove that $w\in A_p(\real^N)$ for some $N\in\mathbb N$. That is, there exists a constant $C>0$ so that for all $x\in\real^N$ and $R>0$ one has
\[
\fint_{B(x,R)} w(y)\,dy \left(\fint_{B(x,R)} w(y)^{1/(1-p)}\,dy\right)^{p-1}\leq C.
\]
Observe that $w^{1/(1-p)}$ is also constant over dyadic annuli. This is easy because
\begin{eqnarray*}
\sup_{R\leq |x|\leq 2R} w(x)^{1/(1-p)}&=& \left(\inf_{R\leq |x|\leq 2R} w(x)\right)^{1/(1-p)} \\&\leq& \left(\beta^{-1}\ \sup_{R\leq |x|\leq 2R} w(x)\right)^{1/(1-p)} \\&=&\beta^{1/(p-1)} \inf_{R\leq |x|\leq 2R} w(x)^{1/(1-p)}.
\end{eqnarray*}
Then we can choose $N$ so that (\ref{hardy.aveces}) holds for $w$ and $w^{1/(1-p)}$ in $\real^N$.
If $|x|\leq 2R$, one has $B(x,R)\subset B_{3R}$ and then
\begin{eqnarray*}
\fint_{B(x,R)} w(x)\,dx && \!\!\!\!\!\!\!\!\!\!\!\!\!\left(\fint_{B(x,R)} w(x)^{1/(1-p)}\,dx\right)^{p-1} \\&\leq& 3^{pn} \fint_{B_{3R}} w(x)\,dx \left(\fint_{B_{3R}}  w(x)^{1/(1-p)}\,dx\right)^{p-1}
\\ &\leq& 3^{pn} 2\beta\ w_0(3R) \left(2\beta^{1/(p-1)}\ w_0(3R)^{1/(1-p)}\right)^{p-1}\\&\leq& (2\,3^{n})^p\beta^2.
\end{eqnarray*}
Assume conversely that $|x|> 2R$. If $y\in B(x,R)$, then $|x|
/2 \leq |y| \leq 3|x|/2$ and consequently
\[
\beta^{-1}w(x)\leq w(y) \leq \beta w(x).
\]
Hence,
\begin{eqnarray*}
\fint_{B(x,R)} w(y)\,dy && \!\!\!\!\!\!\!\!\!\!\!\!\!\left(\fint_{B(x,R)} w(y)^{1/(1-p)}\,dy\right)^{p-1} \\&\leq& \beta w(x) \left((\beta^{-1} w(x)\,)^{1/(1-p)}\right)^{p-1}\leq \beta^2.
\end{eqnarray*}
\end{proof}

\bigskip

\begin{proof}[Proof of Lemma \ref{lema.ecoda.hardy}]
Assuming that we are in $\real^N$ with $2^N>\beta$ we have
\begin{eqnarray*}
\mu(B_R) &=& \omega_{N-1}\int_0^R w_0(t) t^{N-1}\,dt= \omega_{N-1} \sum_{j=0}^{\infty} \int_{2^{-j-1}R}^{2^{-j}R} w(t)t^{N-1}\,dt\\ &\leq& \omega_{N-1} \sum_{j=0}^{\infty} \int_{2^{-j-1}R}^{2^{-j}R} \beta^{j+1} w_0(R)\, t^{N-1}\,dt \\&\leq&
\beta\ w_0(R)\frac{\omega_{N-1}R^N}N \sum_{j=0}^{\infty} \left(\frac\beta{2^N}\right)^{j}
= \frac{\beta}{1-\beta/2^{N}}\ w_0(R)\ |B_R|.
\end{eqnarray*}
Taking $N$ big enough one has that $1/(1-\beta/2^N)\leq 2$, as stated.
\end{proof}

\bigskip

\bigskip

For the proof of Theorem \ref{teorema.radiales}, we will follow \cite{MenarguezSoria2}. First, associated with a weight $v$ on $\real$ we define the non-centered maximal function
\begin{equation}\label{maximal.no.centrado}
\tilde M_{v}F(x)=\sup_{a\leq x\leq b}\frac1{v([a,b])}\int_{[a,b]}|F(t)|v(t)\,dt,
\end{equation}
for each $F$ locally integrable  with respect to $v$. Then, using a simple covering argument for intervals in $\real$ one has the estimate
\begin{equation}\label{desigualdad.unidimensional}
v(\{r\geq0\,:\,\tilde M_v F(r)>\lambda\})\leq \frac2\lambda \int_\real |F(r)|v(r)\,dr,
\end{equation}
(see \cite{MuckenhouptStein}, \cite{Garsia} or \cite{MenarguezTesis}).

\bigskip

We make now the following definition. Given a measurable set $E\in\real^n$ we define its projection onto the sphere $\mathbb S^{n-1}$ by
\[
\Sigma_E = \{\theta\in\mathbb S^{n-1}\ :\ r\theta \in E\ \mbox{for some}\ r>0\}.
\]
The following geometrical result can be found in \cite{MenarguezSoria2}.

\bigskip

\begin{lema}\label{lema.conjunto.radial}
For each ball $B(x,R)\in\real^n$ there exists a set $D$ such that
\begin{enumerate}[(a)]
\item\label{cond1} $B(x,R)\subset D$,
\item\label{cond2} $\Sigma_D=\Sigma_{B(x,R)}$,
\item\label{cond3} for each $\theta\in\Sigma_D$ there exist $0\leq a_\theta\leq b_\theta$ such that $r\theta\in \bar{D}$ if and only if $a_\theta\leq r\leq b_\theta$,
\item\label{cond4}$|x|\theta\in \bar{D}$ for each $\theta\in\sigma_D$ (this means $a_\theta\leq|x|\leq b_\theta$),
\item\label{cond5} $D$ is contained in the union of $B(x,R)$ and another ball $B(z,R)$ with $z\in B(x,R)$.
\end{enumerate}
\end{lema}

\bigskip

Using this, we prove Theorem \ref{teorema.radiales}.

\begin{proof}[Proof of Theorem \ref{teorema.radiales}.]
By conditions \emph{(\ref{cond1})} and \emph{(\ref{cond5})} and the hypothesis that $\mu$ is weak\-ly doubling, we have
\begin{eqnarray*}
\frac1{\mu(B(x,R))}\int_{B(x,R)} |f(y)|\,d\mu(y) &\leq& \frac{\mu(D)}{\mu(B(x,R))} \frac1{\mu(D)} \int_D |f(y)|\,d\mu(y) \\ &\leq& \frac{1+K_1}{\mu(D)} \int_D |f(y)|\,d\mu(y).
\end{eqnarray*}
Assume that $f(x)=f_0(|x|)$ and set $v(t)=w_0(t)t^{n-1}$. Now we integrate along each ray coming from the origin and use conditions \emph{(\ref{cond2})}--\emph{(\ref{cond4})}
\begin{eqnarray*}
\int_D |f(y)|\,d\mu(y) &=& \int_{\Sigma_D}\int_{a_\theta}^{b_\theta} |f_0(t)|v(t)\,dt\,d\sigma_n(\theta)\\&=& \int_{\Sigma_D}\frac{v([a_\theta,b_\theta])}{v([a_\theta,b_\theta])}\int_{a_\theta}^{b_\theta} |f_0(t)|v(t)\,dt\,d\sigma_n(\theta)\\&\leq& \int_{\Sigma_D} v([a_\theta,b_\theta])\,\tilde M_{v}f_0(|x|)\,d\sigma_n(\theta).
\end{eqnarray*}
Note that
\[
\int_{\Sigma_D} v([a_\theta,b_\theta])\, d\sigma_n(\theta) = \int_{\Sigma_D} \int_{a_\theta}^{b_\theta} v(t)\,dt\, d\sigma_n(\theta)=\mu(D),
\]
Hence we have proved that
\[
\frac1{\mu(B(x,R))}\int_{B(x,R)} |f(y)|\,d\mu(y) \leq (1+K_1)\,\tilde M_{v}f_0(|x|),
\]
and, since $R$ is arbitrary, $M_\mu f(x)\leq (1+K_1)\tilde M_v f_0(|x|)$. Integrating in polar coordinates we have
\[
\mu\left(\{x\in\real^n\ :\ Mf(x)>\lambda\}\right) \leq  \omega_{n-1}\,v\left(\left\{r\geq0\,:\,\tilde M_v f_0(r)>\frac{\lambda}{1+K_1}\right\}\right).
\]
By (\ref{desigualdad.unidimensional}) the latter term is bounded by
\[
\frac{2(1+K_1)}\lambda\,\omega_{n-1}\int_0^\infty |f_0(t)|\,v(t)\,dt = \frac{2(K_1+1)}\lambda\int_{\real^n}|f(x)|\,d\mu(x).
\]
\end{proof}

\bigskip

\subsection{Differentiation through dimensions}\label{sect.diferenciacion}

\bigskip

In this section we introduce a new technique that will be useful in settling several questions mentioned in this work, among them, the proof of Theorem \ref{teo.ecoda.wdoubling}. It starts with the following observation: take, in each Euclidean space $\mathbb R^n$,  a ball $B^n$ of a fixed radius $R$ and with center at a fixed distance $s$ from the origin. If $w_0 \in L^1_{\loc}([0,\infty), t^{N-1}dt)$, for some $N\ge 1$,  and $m_n$ denotes Lebesgue measure in $\mathbb R^n$, then the limits
$$
\lim_{n\rightarrow \infty}   \fint_{B^n} w_0(|x|)\, dm_n(x)
$$
exist a.e. depending on $R$ and $s$. Observe that our integrability condition on $w_0$ ensures that the function $w(x)=w_0(|x|)$ is locally integrable in each $\real^n $ whenever $n\ge N$. Let us remark at this point  that for balls $B^n$ of the same radius one always has $\lim_{n\rightarrow \infty} m_n(B^n)=0$. This is not however the reason for this phenomenon that we will refer to as \lq\lq differentiation through dimensions". The precise statement of this type of differentiation is contained in the following result.

\bigskip

\begin{lema}\label{lema.differentiation}
Take $w_0\in L_{\loc}^1([0,\infty),t^{N-1}\,dt)$ for some $N\geq 1$. Then, for almost every $T>0$ and for all $s\geq 0$ and $R>0$ so that $s^2+R^2=T^2$, if we take points $z^n \in \real^n$ with $|z^n|=s$ and we denote $B(z^n,R)=\{y\in\real^n: |z^n-y|< R\}$, the following holds
\[
\lim_{n\rightarrow\infty} \fint_{B(z^n,R)} w_0(|x|)\, dm_n(x) = w_0(T).
\]
\end{lema}

\bigskip

\begin{proof}
The idea is to exploit the fact that in high dimensions the measure of a ball concentrates around `maximal circles'. We will assume that $w_0$ is positive.

\bigskip

Fix $T>0$, and take positive $s$, $z^n\in \real^n$ with $|z^n|=s$ and $R$ with $s^2+R^2=T^2$. Observe that
\[
\fint_{B(z^n,R)} w_0(|x|)\, dm_n(x) = \frac{n}{\omega_{n-1}R^n} \int_{(s-R)_+}^{s+R} w_0(t) A_n(t) t^{n-1} \,dt.
\]
where $A_n(t)=A_n^{(s,R)}(t)=\left|\{\theta\in \mathbb S^{n-1}\,:\, t\theta\in B(z^n,R)\}\right|_{n-1}$. Define
\[
\varphi_n(t)=\varphi_n^{(s,R)}(t)=\frac{n}{\omega_{n-1}R^n} A_n(t) t^{n-1} \chi_{[(s-R)_+,s+R]}(t).
\]
With this notation we have
\[
\fint_{B(z^n,R)} w_0(|x|)\, dm_n(x) = \int_{-\infty}^\infty w_0(t) \varphi^{(s,R)}_n(t)\,dt=: \Phi_n^{(s,R)}w_0.
\]
For a continuous $w_0$ the proof follows if we show that $\varphi_n\tiende^{n\rightarrow\infty}\delta_{T}$ in the sense of distributions. For this it is enough to check that $\varphi_n$ is an approximation of the identity at the point $t=T$, that is
\begin{enumerate}[i)]
\item\label{point.i} $\varphi_n\geq 0$ for all $n\in\mathbb N$,
\item\label{point.ii} $\int_{-\infty}^\infty \varphi_n(t)\,dt=1$ for all $n\in \mathbb N$,
\item\label{point.iii} $\forall \eps>0$ one has $I_\eps(n)=\int_{|t-T|>\eps} \varphi_n(t)\,dt \tiende^{n\rightarrow\infty} 0$.
\end{enumerate}
Note that \ref{point.i}) is trivial and \ref{point.ii}) is immediate from the observation that
\[
\omega_{n-1}\int_{(s-R)_+}^{s+R} A_n(t) t^{n-1}\,dt= |B(z^n,R)|_n.
\]
To see \ref{point.iii}) observe that $I_\eps(n) = I_\eps^1(n) + I_\eps^2(n)$ where
\begin{eqnarray*}
I_\eps^1(n) &=& \int_{-\infty}^{T-\eps} \varphi_n(t)\,dt = \frac{|B(z^n,R)\cap B_{T-\eps}|_n}{|B(z^n,R)|_n}\leq \left(\frac{R_1}R\right)^n\ \tiende^{n\rightarrow\infty} 0, \\ I_\eps^2(n) &=& \int_{T+\eps}^\infty \varphi_n(t)\,dt =\frac{|B(z^n,R)\setminus B_{T+\eps}|_n}{|B(z^n,R)|_n}\leq \left(\frac{R_2}R\right)^n\ \tiende^{n\rightarrow\infty} 0,
\end{eqnarray*}
and $R_1$ and $R_2$ are the radii of the minimal balls that contain $B(z^n,R)\cap B_{T-\eps}$ and $B(z^n,R)\setminus B_{T+\eps}$ respectively. It is obvious that $R_1, R_2<R$. This proves Lemma \ref{lema.differentiation}, $\forall T>0$,  if   $w_0$ is continuous in $[0,\infty)$.

\bigskip

For a general $w_0$ we have to show that, if $P_T=\{(s,R)\,:\ s\geq 0,\,R>0,\, s^2+R^2=T^2\}$, the set $E_0(w_0)$ defined as
\[
\left\{T\in [\varepsilon,T_0]\,:\, \exists (s,R)\in P_T\,:\, \limsup_{n\rightarrow\infty} \Phi_n^{(s,R)}w_0 - \liminf_{n\rightarrow\infty} \Phi_n^{(s,R)}w_0>0\right\},
\]
has measure $0$ for every $\varepsilon, T_0$ with $0<\varepsilon<T_0$. 

\bigskip

To that end, we fix $T\in [\varepsilon,T_0]$. For $t\in [(s-R)_+,s+R]$ we denote by $y(t)$  the diameter of the set $\partial B_t \cap B(z^n,R)$. Observe that $y(t)\leq R$ and that $y$ increases with $t$ up to the point $t=T$, where it attains its maximum, and then decreases in $(T, s+R)$. Also, for each $t\in [|s-R|,s+R]$ we call $\alpha(t)$  the angle between the segment connecting the origin with $z^n$ and the one joining the origin with any point in $\partial B_t \cap \partial B(z^n,R)$. Clearly $y(t)=t\sin \alpha(t)$.

\bigskip

Observe first that the function $\varphi_n^{(s,R)}$ is decreasing in the interval $[T,s+R]$. To see this note that
\begin{equation}
\frac{A_n(t)t^{n-1}}{\omega_{n-2}} = t^{n-1}\int_0^{\alpha(t)}(\sin \beta)^{n-2}\,d\beta = \int_0^{y(t)} u^{n-2}\,\frac{du}{\sqrt{1-(u/t)^2}},
\end{equation}
and that, both, the integrand and  $y(t)$  decrease with $t$ in this interval. Hence,
\[
\int_T^{s+R}\varphi_n^{(s,R)}(t)w_0(t)\,dt\leq M_{\rightarrow}\overline w_0(T),
\]
where $M_{\rightarrow}$ denotes the one-sided maximal operator
\[
M_{\rightarrow}f(t)=\sup_{h>0}\frac1h\int_t^{t+h} |f(s)|\,ds,
\]
and $\overline w_0$ is the restriction of $w_0$ to the interval $[\varepsilon,2T_0]$.

\bigskip

Let us assume that $s>0$. If $\frac \pi 2\le \alpha(t)\le \pi$, we will use that 
\begin{equation} \label{one} t^{1-N}\,\varphi_n^{(s,R)}(t)\le \frac{nt^{n-N}}{R^n}.\end{equation}
This case may happen only if $R>s$ and, then, $t\le \sqrt{R^2-s^2}<R.$ For $\alpha(t)<\frac \pi 2$ we have the estimate
\begin{equation} \label{two} t^{1-N}\,\varphi_n^{(s,R)}(t)\le \frac{nt^{n-N}}{\omega_{n-1}R^n}\, \omega_{n-2} \, \alpha(t)\, \sin^{n-2}\alpha(t) \le \frac{n \omega_{n-2}}{\omega_{n-1}R^n}\,  \, \frac \pi 2\, y(t)^{n-N}  .\end{equation}
Combining (\ref{one}) and (\ref{two}) we see that for all $t_0<T$ we have 
\[
\lim_{n\rightarrow\infty}\left[\sup_{t<t_0} t^{1-N}\,\varphi_n^{(s,R)}(t)\right] =0.
\]
As a consequence, if $t_0<T$
\begin{eqnarray*}
\limsup_{n\rightarrow\infty}\int_{0}^{t_0} \varphi_n^{(s,R)}(t)\ w_0(t)\,dt &\le& 
\lim_{n\rightarrow\infty}\left[\sup_{t<t_0} t^{1-N}\,\varphi_n^{(s,R)}(t)\right] \int_{0}^{t_0}  w_0(t)t^{N-1}\,dt\\&=&0.
\end{eqnarray*}
Still for $\alpha(t)<\frac \pi 2$, we make the observation that 
$$
A_n(t)=\omega_{n-2}  \int_0^{\alpha(t)}(\sin \beta)^{n-2}\,\cos \beta \, \cos^{-1} \beta\, d\beta,
$$
and so, $\frac{\omega_{n-2}}{n-1}\, \sin^{n-1}\alpha(t)\le A_n(t)\le 
\frac{\omega_{n-2}}{n-1}\, \sin^{n-1}\alpha(t)\,\cos^{-1}\alpha(t). $ Thus, if we call 
\[
\tilde \varphi_n^{(s,R)}(t) = \frac{n}{n-1} \frac{\omega_{n-2}}{\omega_{n-1}}\frac{y(t)^{n-1}}{R^n},
\]
then we have the double estimate
$$\tilde \varphi_n^{(s,R)}(t)\leq \varphi_n^{(s,R)}(t) \leq \tilde \varphi_n^{(s,R)}(t)\cdot \cos^{-1}\alpha(t).
$$
It is important to point out that $\tilde \varphi_n^{(s,R)}(t)$ is increasing in $[0,T]$ and that 
$$\int_{[0,T]}\tilde\varphi_n^{(s,R)}(t)\, dt \le \int_{[0,T]}\varphi_n^{(s,R)}(t)\, dt \le 1.
$$ 
The final observation is that there exists $t_1=t_1(s,R)<T$, so that $\alpha(t)$ decreases in $[t_1,T]$. Hence, for $t_0\in [\max \{t_1,\,\varepsilon\}, \, T)$ we have 
\begin{eqnarray*}
\int_{t_0}^T \varphi_n^{(s,R)}(t)w_0(t)\,dt&\leq& \int_{t_0}^T \tilde\varphi_n^{(s,R)}(t)w_0(t)\,\cos^{-1}\alpha(t)\,dt  \\&\leq& {\cos^{-1}\alpha(t_0)} {M_{\leftarrow}\overline w_0(T)} ,
\end{eqnarray*}
where $M_\leftarrow$ is the maximal operator
\[
M_{\leftarrow}f(t)=\sup_{h>0}\frac1h\int_{t-h}^t |f(s)|\,ds.
\]
For this $t_0$,
\begin{eqnarray*}
\limsup_{n\rightarrow\infty} \Phi_n^{(s,R)}w_0 &=&\limsup_{n\rightarrow\infty}
\left(\int_0^{t_0}  \varphi_n^{(s,R)}(t)w_0(t)\,dt+\int_{t_0}^T  \varphi_n^{(s,R)}(t)w_0(t)\,dt\right)\\[2mm]& \leq & {\cos^{-1}\alpha(t_0)} {M_{\leftarrow}\overline w_0(T)}.
\end{eqnarray*}
Taking the limit as $t_0 \rightarrow T^-$, we have 
\begin{eqnarray*}
\limsup_{n\rightarrow\infty} \Phi_n^{(s,R)}w_0 &\leq& 
 {\cos^{-1}\alpha(T)} {M_{\leftarrow}\overline w_0(T)}=\frac Ts \, {M_{\leftarrow}\overline w_0(T)}\\ &\leq& \left(1+\frac Rs\right)\, M_{\leftarrow}\overline w_0(T).
\end{eqnarray*}
Therefore, if $R/s\leq k$ for $k\in\mathbb N$ we have
\[
\limsup_{n\rightarrow\infty} \Phi_n^{(s,R)}w_0 \leq M_{\rightarrow}\overline w_0(T)+(1+k) M_{\leftarrow} \overline w_0(T) \leq 3k \tilde M\overline w_0(T),
\]
with $\tilde M$ denoting, as in (\ref{maximal.no.centrado}), the non centered one dimensional maximal operator
\[
\tilde M f(t)=\sup_{a\leq t\leq b} \frac1{b-a}\int_a^b |f(s)|\,ds.
\]
Observe that if $s=0$, then $\varphi_n^{(0,T)}(t)$ is an increasing function in the interval $[0,T]$. Arguing as before to remove the integration on $[0,\varepsilon]$, we have  
\[
\limsup_{n\rightarrow\infty} \Phi_n^{(0,T)}w_0\leq M_\leftarrow \overline w_0(T) \leq \tilde M\overline w_0(T).
\]

\bigskip

To finish the proof, for $T>0$ and $k\in\mathbb N$ we denote $P_T^k=\{(s,R): s,R>0,\ R/s\leq k,\  s^2+R^2=T^2\}\cup \{(0,T)\}$ and we define for $\lambda>0$ the set $E_\lambda^k(w_0)$ as
\[
\left\{T\in [\varepsilon,\, T_0]\,:\, \exists (s,R)\in P_T^k\,:\, \limsup_{n\rightarrow\infty} \Phi_n^{(s,R)}w_0 - \liminf_{n\rightarrow\infty} \Phi_n^{(s,R)}w_0>\lambda\right\}.
\]
Note that if $g$ is continuous, the previous considerations and  (\ref{desigualdad.unidimensional}) give
\begin{eqnarray*}
\left|E_\lambda^k(w_0)\right| &=& \left|E_\lambda^k(w_0-g)\right| \leq \left|\left\{ T>0\,:\, 6k \tilde M (\overline w_0-\overline g)(T)>\lambda\right\}\right|\\&\leq & \frac{12k}\lambda \|\overline w_0-\overline g\|_{L^1},
\end{eqnarray*}
This last term can be made as small as needed with an appropriate choice of $g$. Therefore $\left|E_\lambda^k(w_0)\right|=0$. As a consequence, $\left|E_0^k(w_0)\right|=\left|\bigcup_{\lambda>0} E_\lambda^k(w_0)\right| =0$ and, hence, $\left|E_0(w_0)\right|=\left|\bigcup_{k\in\mathbb N} E_0^k(w_0)\right|=0$, as wanted.

\end{proof}

\bigskip

\subsection{Some applications of the differentiation through dimensions}\label{sect.proof3.3}

As a consequence of this new technique of differentiation, we prove here Theorem \ref{teo.ecoda.wdoubling} and show the equivalence between the two properties of weakly doubling and strong $n$-micro-doubling when they hold uniformly with the dimension.

\begin{proof}[Proof of Theorem \ref{teo.ecoda.wdoubling}]
First suppose that $\mu$ is uniformly weakly doubling with constant $K_1$ in each $\real^n$ with $n\geq N$ for some $N\in \mathbb N$. We have to prove that $w$ is essentially constant over dyadic annuli.

\bigskip

With $P_T$ as in the previous proof, let
\[
\mathcal{T}=\left\{ T>0\,:\,  \lim_{n\rightarrow\infty} \frac{\mu(B(z^n,R))}{|B(z^n,R)|}=w_0(T), \quad \forall (|z^n |,R)\in P_T \right\}.
\]
We know from Lemma \ref{lema.differentiation} that $\mathcal{T}$ has full measure in $(0,\infty)$.

\bigskip

Now, for $R\in \mathcal{T}$ and $T\in \mathcal{T}\cap [R,2R]$ we take $z^n\in \real^n$ with $|z^n|^2+R^2=T^2$. Observe that $0\leq|z^n|\leq \sqrt{(2R)^2-R^2}=\sqrt3 R$, so that $B_R\cap B(z^n,R)\neq \emptyset$. By the hypothesis on $\mu$
\[
\frac1{K_1} \frac{\mu(B_R)}{|B_R|} \leq \frac{\mu(B(z^n,R))}{|B_R|} \leq K_1\, \frac{\mu(B_R)}{|B_R|}.
\]
Taking limits when $n\rightarrow\infty$, as we may since $R, T\in \mathcal{T}$ we get
\begin{equation}\label{after.differentiation}
\frac1{K_1}\ w_0(R)\leq w_0(T) \leq K_1\ w_0(R).
\end{equation}
This shows that $w_0$ is essentially constant on dyadic intervals with constant $\beta=K_1$.

\bigskip

The reverse implication is straightforward and will be omitted.
\end{proof}

\bigskip

We finish this section with a proof of the following important equivalence.

\bigskip

\begin{teo}\label{teo.ecoda.strong} A rotation invariant measure $\mu$ is uniformly strong $n$-micro-dou\-bling in each $\real^n$ for $n\geq N$ if and only if $\mu$ is uniformly weakly doubling.
\end{teo}

\bigskip

\begin{proof}
In view of Theorem \ref{teo.ecoda.wdoubling} we only need to prove that a measure given by a radial density that is essentially constant over dyadic annuli, is also uniformly $n$-micro-doubling.

\bigskip

Assume that $w$ is essentially constant over dyadic annuli with constant $\beta$ and take $n\geq N$ for the $N\in\mathbb N$ obtained in Lemma \ref{lema.ecoda.hardy}. Given $x\in\real^n$ and $R>0$, we will write $B^\ast(x,R)=B(x,(1+1/n)R)$. We first consider the case  $|x|\geq 3R$, which implies that $(1+1/n)R\leq 2|x|/3$. If $y\in B^\ast(x,R)$, then $|x|/3\leq |y|\leq 5|x|/3$ which means that $\beta^{-2}\,w(x)\leq w(y)\leq \beta^2 w(x)$. From this we get
\begin{eqnarray*}
\mu\bigl(B^\ast(x,R)\bigr) &=& \int_{B^\ast(x,R)} w(y)\,dy \leq \beta^2\ w(x)\ |B^\ast(x,R)|\\ &\leq& e\beta^2\ w(x)\ |B(x,R)|\leq  e\beta^4\ \mu(B(x,R)).
\end{eqnarray*}

\bigskip

If $|x|\leq 3R$ we split $B^\ast (x,R)$ into two disjoint pieces. For the one intersecting $B_{R/2}$ we use Lemma \ref{lema.ecoda.hardy} as follows
\[
\mu(B^\ast(x,R)\cap B_{R/2}) \leq \mu(B_{R/2}) \leq 2\beta\ w_0(R/2)\ |B_{R/2}| \leq \beta^2\ w_0(R)\ |B_R|.
\]
In the complementary piece $w$ is essentially constant: if $y\in B^\ast(x,R)\setminus B_{R/2}$, then $R/2\leq |y| \leq 5R$, which means that $\beta^{-3}w_0(R)\leq w(y)\leq \beta^3\,w_0(R)$. Hence
\[
\mu(B^\ast(x,R)\setminus B_{R/2}) \leq \beta^3\,w_0(R)\ |B^\ast(x,R)| \leq e\beta^3\,w_0(R)\ |B_R|.
\]
On the other hand note that if $y\in B(x,R)\setminus B_{R/2}$ then $R/2\leq|y|\leq 4R$ and consequently $\beta^{-2}w_0(R)\leq w(y)\leq \beta^2\,w_0(R)$. This yields
\begin{eqnarray*}
\mu(B(x,R))&\geq& \mu(B(x,R)\setminus B_{R/2}) \geq \frac1{\beta^{2}} \,w_0(R)\ |B(x,R)\setminus B_{R/2}| \\&\geq& \frac1{2\beta^{2}}\,w_0(R)\ |B_R|.
\end{eqnarray*}
\end{proof}

\bigskip

\begin{xrem}
Observe that the only thing that we have really proved here is that a uniform weakly doubling property implies a uniform micro-doubling property. The reciprocal is not true. In the last Section, we show an example of a density for which the associated measure is uniformly $n$-micro-doubling but is not uniformly weakly doubling.
\end{xrem}

\bigskip

\subsection{Decreasing densities.}\label{sect.decrecientes}

\bigskip

Radial measures with decreasing densities have some interesting and sometimes surprising properties. For instance, to ensure that a measure with such a density is uniformly weakly bounded, it is enough to check this condition just in one concrete Euclidean space.

\bigskip

\begin{prop}\label{prop.dec.unif.wdoubling}
Let $w_0$ be a decreasing function over $[0,\infty)$. If $\mu$ is weakly doubling in $\real^N$ for some $N\in\mathbb N$, then $\mu$ is uniformly weakly doubling in $\real^n$ for all $n\geq N$.
\end{prop}

\bigskip

\begin{proof} Assume that $\mu$ is weakly doubling with constant $K_1$ in $\real^N$ for some $N\in\mathbb N$. By Theorem \ref{teo.ecoda.wdoubling} we only need to show that $w$ is essentially constant over dyadic annuli. Given $R>0$, if $R\leq |x|\leq 2R$, since $w_0$ is decreasing, one has $w_0(2R)\leq w(x) \leq w_0(R)$. Thus, we only need to check that $w(R)\leq Cw(2R)$ with $C$ independent of $R$. To do so, consider a point $z\in\real^N$ so that $|z|=R/2$. Take $z'=3z$ and $z''=5z$ and consider the balls $B=B(z,R/2)$, $B'=B(z',R/2)$ and $B''=B(z'',R/2)$.

\bigskip

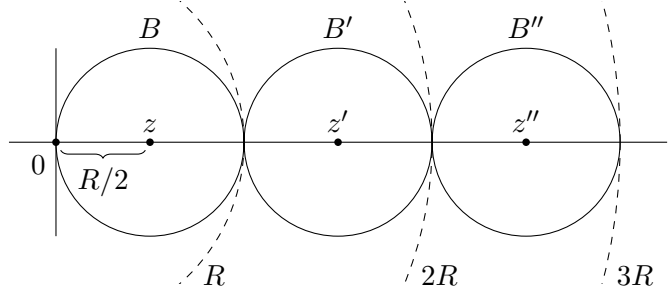
\begin{center}
\begin{tikzpicture}

\def\R{2.5}
\pgfmathsetmacro{\Rm}{\R/2}

\coordinate (O) at (0,0);
\coordinate (Z0) at (\Rm,0);
\coordinate (Z1) at (3*\Rm,0);
\coordinate (Z2) at (5*\Rm,0);
\coordinate (Y) at (0,\Rm);
\coordinate (Ya) at ($1.5*(Y)$);

\draw (0,-\Rm)--(0,\Rm);
\draw (-\Rm/2,0)--(6.5*\Rm,0);

\draw (Z0) circle (\Rm) (Z1) circle (\Rm) (Z2) circle (\Rm);

\fill (O) node [left,yshift=-3mm] {$0$} circle (0.5mm) (Z0) node [above] {$z$} circle (0.5mm) (Z1) node [above] {$z'$} circle (0.5mm) (Z2) node [above] {$z''$} circle (0.5mm) ;

\draw ($(Z0)+(Y)$) node [above] {$B$} ($(Z1)+(Y)$) node [above] {$B'$} ($(Z2)+(Y)$) node [above] {$B''$};

\begin{scope}
	\clip (Ya) -- ($(Ya)+(6.5*\Rm,0)$) -- ($(6.5*\Rm,0)-(Ya)$) -- ($(O)-(Ya)$) -- cycle;
	\draw [dashed] (O) circle (\R) (O) circle (2*\R) (O) circle (3*\R);
\end{scope}

\draw (\R,-1.2*\Rm) node [below,xshift=-4mm] {$R$} (2*\R,-1.2*\Rm) node [below,xshift=1mm] {$2R$} (3*\R,-1.2*\Rm) node [below,xshift=2mm] {$3R$};

\draw [decorate,decoration={brace,amplitude=3pt,raise=3pt},yshift=0pt]
($0.05*(O)+0.95*(Z0)$)--($0.95*(O)+0.05*(Z0)$) node [midway,below,xshift=0mm,yshift=-2mm] {$R/2$};

\end{tikzpicture}
\end{center}

\bigskip

By the weakly doubling property of $\mu$ we have that $\mu(B)\leq K_1 \mu(B') \leq K_1^2 \mu(B'')$. For all $y\in B$ we have $w_0(R)\leq w(y)$ and for all $y\in B''$ we have $w(y)\leq w_0(2R)$, this yields
\[
w_0(R) \leq \frac{\mu(B)}{|B|} \leq K_1^2\ \frac{\mu(B'')}{|B''|} \leq K_1^2\ w_0(2R).
\]
\end{proof}

\bigskip

Also, if the density is decreasing, the reciprocal implication in Theorem \ref{teorema.radiales} is true.

\bigskip

\begin{prop}\label{prop.reciproco.orejas}
Let $w_0$ be a decreasing function over $[0,\infty)$. Assume that there exists $C>0$ so that for some $N>0$ one has $\|M_\mu\|_{L_{\rad}^1(\real^n,d\mu)\rightarrow L_{\rad}^{1,\infty}(\real^n,d\mu)} \leq C$ for all $n\geq N$. Then $\mu$ is uniformly weakly doubling in $\real^n$ for all $n\geq N$. 
\end{prop}

\bigskip

Observe that by Theorem \ref{teo.wdoublin.unif}, the hypothesis in the previous proposition also implies that $M_\mu$ is uniformly bounded on $L^p(\real^n,d\mu)$ for all $n\geq N$. It is remarkable that a boundedness condition over radial functions implies one for general functions. The proof of this Proposition is based on differentiation through dimensions.

\bigskip

\begin{proof}
Assume that $M_\mu$ satisfies the weak $L^1(\mu)$ inequality
\[
\mu\left(\{x\in\real^n: M_\mu f(x)>\lambda\}\right) \leq \frac {C_\ast}\lambda \int_{\real^n}|f(x)|\,d\mu(x),
\]
for all $f\in L_{\rad}^1(\real^n)$ and $n\geq N$ for certain $N\in\mathbb N$, with $C_\ast$ independent of $f$ and $n$. As before, it is enough to show that $w$ is essentially constant over dyadic annuli. Since $w_0$ is decreasing, it suffices to prove that $w_0(R)\leq C w_0(2R)$ with $C$ independent of $R$. Fix $R>0$ and for each $n\geq N$, take $z^n\in\real^n$ so that $|z^n|=R$. Set $f=\chi_{B_\eps}$ with $0<\eps<R$ and
\[
\lambda=\frac{\mu(B(z^n,R+\eps))}{\mu(B_\eps)}.
\]
Clearly we have $B_R\subset \{x\in\real^n:\ M_\mu f(x)>\lambda\}$, which together with our hypothesis gives
\[
\mu(B_R)\leq \frac {C_\ast}\lambda \int_{\real^n} |f(x)|\,d\mu(x) = C_\ast\ \mu(B(z^n,R+\eps)).
\]
Letting $\eps\rightarrow 0^+$, by monotonicity we obtain $
\mu(B_R)\leq C_\ast \mu(B(z^n,R))$. In particular we have
\[
\frac{\mu(B_R)}{|B_R|} \leq C_\ast \frac{\mu(B(z^n,R))}{|B(z^n,R)|},
\]
and now we differentiate through dimensions. Letting $n\rightarrow\infty$ in the last inequality, Lemma \ref{lema.differentiation} gives $w_0(R) \leq C_\ast w_0(\sqrt 2 R)$ from which we deduce that $w_0(R)\leq C w_0(2R)$ with $C=C_\ast^2$.
\end{proof}

\bigskip

We now make a connection between the previous hypotheses and the Muckenhoupt $A_1$ condition. We recall that for a positive $w\in L_{\loc}^1(\real^n)$ the $A_1$-constant we call $[w]_{A_1(\real^n)}$ is defined as the smallest constant $C$ so that
\[
\fint_{B(x,R)} w(y)\,dy \leq C \essinf_{y\in B(x,R)} w(y),
\]
for all $x\in \real^n$ and $R>0$. Under the hypothesis that $w_0$ is decreasing in $[0,\infty)$ and $d\mu(x)=w_0(|x|)\,dx$ is uniformly weakly doubling we will show that $w(x)=w_0(|x|)\in A_1(\real^n)$ for all sufficiently large $n$. Moreover, the $A_1$-constants can be bounded independently of $n$. The precise statement is the following.

\bigskip

\begin{prop}
Let $w_0$ be a decreasing function over $[0,\infty)$ and essentially constant over dyadic intervals with constant $\beta$. Then there exist $N>0$ and $C>0$ only depending on $\beta$ so that $[w]_{A_1(\real^n)}\leq C$ for all $n\geq N$.
\end{prop}

\bigskip

\begin{proof}
Assume that $w$ is essentially constant over dyadic intervals with constant $\beta$ and take $n\geq N$ for the $N$ given in Lemma \ref{lema.ecoda.hardy}. We have to find a constant $C>0$ so that for all $x\in\real^n$ and $R>0$ one has
\[
\frac{\mu(B(x,R))}{|B(x,R)|} \leq C\inf_{y\in B(x,R)} w(y).
\]
If $|x|\leq 2R$, since $w$ is decreasing and using Lemma \ref{lema.ecoda.hardy} and that $w$ is essentially constant over dyadic annuli one has
\[
\frac{\mu(B(x,R))}{|B(x,R)|}\leq \frac{\mu(B_R)}{|B_R|} \leq 2\beta\ w_0(R)\leq 2\beta^3\ w_0(4R).
\]
This finishes the proof in this case because now using again that $w_0$ is decreasing we obtain
\[
w_0(4R)\leq \inf_{y\in B_{4R}}w(y) \leq \inf_{y\in B(x,R)} w(y).
\]
If $|x|\geq 2R$ then $w$ is essentially constant over $B(x,R)$. Indeed if $y\in B(x,R)$ then $|x|/2\leq |y|\leq 3|x|/2$, which implies that $\beta^{-1}w(x)\leq w(y)\leq \beta\,w(x)$. Hence,
\[
\frac{\mu(B(x,R))}{|B(x,R)|} \leq \beta\ w(x) \leq \beta^2 \inf_{y\in B(x,R)} w(y).
\]
\end{proof}

\bigskip

The amazing thing here is that the reciprocal statement of this last proposition is also true.

\bigskip

\begin{prop}\label{prop.a1.decreciente}
Let $w_0$ be a non-negative function over $[0,\infty)$ and set $w(x)=w_0(|x|)$ so that $w\in A_1(\real^n)$ for $n\geq N>0$. If there exists $C_\ast>0$ so that $[w]_{A_1(\real^n)}\leq C_\ast$ for all $n\geq N$, then $w_0$ is essentially constant over dyadic intervals and comparable with a decreasing function.
\end{prop}

\bigskip

\begin{proof}[Proof of Proposition \ref{prop.a1.decreciente}]
It is enough to prove that $w$ is comparable with a decreasing function $\tilde w$ because then $\tilde w\in A_1$ and, therefore, $\tilde w$ is (weakly) doubling. By Proposition \ref{prop.dec.unif.wdoubling} and Theorem \ref{teo.ecoda.wdoubling} we would have that $\tilde w$ is essentially constant over dyadic intervals, and hence so is $w$. 

\bigskip

It suffices to check that there is a constant $q>0$ so that $w_0(t)\leq q\,w_0(s)$ whenever $0\leq s\leq t$. For $0<s<t$ and $n\geq N$, take $x_n,y_n\in \real^n$ so that $x_n$ is a multiple of $y_n$ and $|x_n|=s$, $|y_n|=t$. Consider the ball $B(x_n,R)$ with $R^2=t^2-s^2$. 

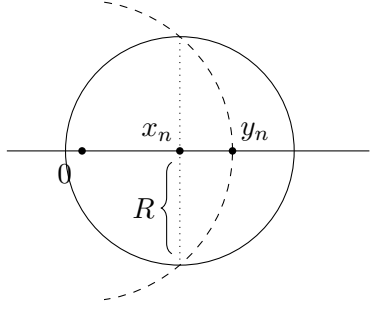
\begin{center}
\begin{tikzpicture}
\def\s{1.3}
\def\t{2}
\pgfmathsetmacro{\Rm}{sqrt(\t*\t-\s*\s)}

\coordinate (O) at (0,0);
\coordinate (X) at (\s,0);
\coordinate (Y) at (\t,0);
\coordinate (T) at (\s-1,0);
\coordinate (F) at (\s+\Rm+1,0);
\coordinate (A) at (0,\Rm+1);
\coordinate (B) at (0,\Rm);

\draw (-1,0)--(F);
\draw (X) circle (\Rm);

\begin{scope}
	\clip ($(T)+(A)$) -- ($(F)+(A)$) -- ($(F)-(A)$) -- ($(T)-(A)$) -- cycle;
	
	\draw [dashed] (O) circle (\t);
\end{scope}

\draw [dotted] ($(X)+(B)$)--($(X)-(B)$);
\draw [decorate,decoration={brace,amplitude=4pt,raise=3pt},yshift=0pt]
($0.1*(X)+0.9*($(X)-(B)$)$)--($0.9*(X)+0.1*($(X)-(B)$)$) node [midway,left,xshift=-2mm,yshift=0mm] {$R$};

\fill (O) node [left,yshift=-3mm] {$0$} circle (0.5mm) (X) node [above,xshift=-3mm] {$x_n$} circle (0.5mm) (Y) node [above,xshift=3mm] {$y_n$} circle (0.5mm);
\end{tikzpicture}
\end{center}

By hypothesis we have, in the almost everywhere sense,
\[
\frac{\mu(B(x_n,R))}{|B(x_n,R)|}\leq C_\ast\ w(x_n)= C_\ast\ w_0(s).
\]
Letting $n\rightarrow\infty$, by the differentiation through dimensions of Lemma \ref{lema.differentiation} we get that
\[
w_0(t)=w_0(\sqrt{s^2+R^2})\leq C_\ast\ w_0(s),
\]
for almost every $s\leq t$.
\end{proof}

\bigskip

\section{Further results and remarks}\label{sect.further}

\bigskip

\subsection{Examples of measures that are not uniformly weakly-doubling}\label{sect.ejemplo}

\bigskip

Each measure $\mu$ with density $w(x)=|1-|x||^{-\alpha}$ for $0<\alpha<1$ is doubling on $\real^n$ for each $n\in\mathbb N$, but is not uniformly weakly doubling. To see this observe that in each $\real^n$ one has $\mu(B_1)\geq C_\alpha n^\alpha |B_1|$. Using differentiation through dimensions, if $|z|=1$, then $\mu(B(z,1))/|B(z,1)|\tiende^{n\rightarrow\infty} w(\sqrt{2})=(\sqrt2-1)^{-\alpha}$. This says that the measures of the intersecting balls $B_1$ and $B(z,1)$ are not comparable with constants independent of the dimension.

\bigskip

Using this idea we can also prove that there are not uniform weak $L^1(\mu)$ bounds for the associated maximal operator.

\bigskip

\begin{teo}\label{teo.weak.l1.no.unif}
Let $\mu$ be the measure defined above and let $C_{1,n}$ be the smallest constant satisfying
\[
\mu(\{x\in\real^n:M_\mu f(x)>\lambda\})\leq \frac{C_{1,n}}\lambda\|f\|_{L^1(\real^n,\mu)},
\]
for all $f\in L^1(\real^n,\mu)$. Then, one has $C_{1,n}\geq cn^\alpha$.
\end{teo}

\bigskip

\begin{proof}
Using discretization (see \cite{Guzman} and \cite{MenarguezSoria}), or the argument in the proof of Proposition \ref{prop.reciproco.orejas}, we may consider $M_\mu$ acting over finite sums of Dirac deltas instead of integrable functions. For $\delta_0 =\lim_{\eps\rightarrow0}\chi_{B_\eps}(x)/\mu(B_\eps)$, in the sense of distributions, the weak $L^1(\mu)$ inequality reads
\begin{equation}\label{desigualdad.delta}
\mu(\{x\in\real^n:M_\mu \delta_0(x)>\lambda\})\leq \frac{C_{1,n}}\lambda.
\end{equation}
Note that $M_\mu\delta_0(x)=1/\mu(B(x,|x|))$ for each $x\neq 0$, which makes $M_\mu\delta_0$ a radially decreasing function. Then taking $\lambda=M_\mu\delta_0(z)$ with $|z|=1$ we have $\{x\in\real^n: M_\mu \delta_0(x)> M_\mu\delta_0(z)\}\subset B_1$ and from (\ref{desigualdad.delta}) we obtain
\[
C_{1,n}\geq \frac{\mu(B_1)}{\mu(B(z,1))}.
\]
This, together with the previous observation that $\mu(B_1) \geq C n^\alpha |B_1|$ and that $\mu(B(z,1))/|B_1| \rightarrow \ w(\sqrt2)$ as $n\rightarrow\infty$, proves the result.
\end{proof}

\bigskip

\begin{xrem}One can prove with some extra work that $\mu$ is uniformly $n$-micro-\-doubling. It is also weakly doubling but in this case with a constant $K_1 \sim n^\alpha$ in each $\real^n$. In particular, Theorems \ref{teo.naor.tao} and \ref{teo.weak.l1.no.unif} imply that $ C_{1,n} \leq c\,n^{1+\alpha}\log n$.\end{xrem}

\bigskip

\subsection{Families of measures changing with the dimension}\label{sect.contraejemplo}

\bigskip

As we have seen in the preceding sections, maximal operators associated with measures given by power densities have $L^p$ operator norms bounded with respect to the dimension. An interesting observation by J.M. Aldaz and J. P\'erez L\'azaro in \cite{AldazPerezLazaro} showed that given an exponent $p$ (as large as wanted) there exist families of power weights such that the $L^p$ bounds of the associated maximal operators grow to infinity as $n\longrightarrow\infty$. The twist here is that the powers change from one dimension to another. To be more precise they considered measures $\nu_{\alpha,n}$ given by the densities $|x|^{-\alpha n}$ over $\real^n$ with $0<\alpha<1$. Their result is the following (see Theorem 3.12 in \cite{AldazPerezLazaro}).

\bigskip

\begin{teo}\label{teorema.medidas.doblantes1}
Given $p_0\in[1,\infty)$, there exist $\alpha_0 \in (0,1)$ and $a>1$ such that for all $p\in[1,p_0]$ and all $\alpha\in[\alpha_0,1)$ one has
\[
c_{\nu_{\alpha,n},p}\geq \frac{a^{(1-\alpha)n}}6.
\]
\end{teo}

\bigskip

It is implicit in the proof given in \cite{AldazPerezLazaro} that $\alpha_0\longrightarrow1$ as $p_0\longrightarrow\infty$. This leads to the question of whether, fixing $\alpha$, $M_{\nu_{\alpha,n}}$ may satisfy a uniform $L^p$ bound for large $p$. We can apply the method used in Theorem \cite{CriadoSjogren} for the Gaussian measure to show that this is not the case when $\alpha>1/2$.

\bigskip

\begin{teo}\label{teo.alphas}
For each $\alpha\in(1/2,1)$ there exists a constant $a>1$ such that for all $p\in[1,\infty)$
\[
c_{\nu_{\alpha,n},p}\geq ca^{n/p},
\]
even if the action is restricted to radially decreasing functions.
\end{teo}
The proof of this result can be found in \cite{tesis} and in \cite{APL}.

\bigskip

A consequence of this result is that for these families of measures the constants of the $n$-micro-doubling and the weak doubling conditions grow to infinity with the dimension (see Theorems \ref{teo.naor.tao}, \ref{teorema.radiales} and \ref{teo.power.lp}).

\end{document}